\documentclass[11pt,reqno,oneside]{amsart}
\usepackage[english, greek]{babel}
\usepackage[T1]{fontenc}
\usepackage[utf8]{inputenc}
\usepackage[mathscr]{eucal}
\usepackage{color}
\usepackage{xcolor}
\usepackage{amsfonts}   
\usepackage{amssymb}
\usepackage{latexsym}
\usepackage{enumerate}
\usepackage[colorlinks,citecolor=purple,linkcolor=blue]{hyperref}

\usepackage{geometry}

\let\nc\newcommand

\nc{\la}{\label}

\newtheorem{theorem}{Theorem}[section]
\newtheorem{corollary}[theorem]{Corollary}
\newtheorem{lemma}[theorem]{Lemma}
\newtheorem{proposition}[theorem]{Proposition}
\theoremstyle{definition}
\newtheorem{definition}[theorem]{Definition}
\newtheorem{example}[theorem]{Example}

\newtheorem{remark}[theorem]{Remark}
\newtheorem*{problem}{Problem}

\def\k{\mathsf k}
\def\C{\mathbb C}

\newcommand{\Frac}{{\rm{Frac}}}

\newcommand{\Autk}{{\rm{Aut}}}

\newcommand{\into}{\,\,\hookrightarrow\,\,}

\newcommand{\onto}{\,\,\twoheadrightarrow\,\,}

\def\k{\mathsf k}

\newcommand{\GK}{{\rm GK}}

\newcommand{\cfs}{\operatorname{cfs}}
\def\M{\mathcal M}

\def\L{\mathcal L}
\def\K{\mathcal K}
\def\Z{\mathbb Z}
\def\m{\mathsf m}
\def\n{\mathsf n}

\newenvironment{dedication}
  {
   \itshape             
   \raggedleft          
  }
  
\begin{document}
\selectlanguage{english}
\title{Harish-Chandra modules and Galois orders revisited}

\author{Jo\~ao Schwarz}
\address{Shenzehn International Center for Mathematics, SUSTech, 1088 Xueyuan Avenue, Shenzhen 518055,
P.R. China}
\email{jfschwarz@hotmail.com}


\maketitle

\begin{dedication}
    {\color{purple} To Pan for all the smiles that she gives me.} 
\end{dedication}

\medskip

\medskip

\medskip

\begin{abstract}
    The main subject of study of this paper are general properties of Harish-Chandra algebras and modules with respect to a pair of algebra and subalgebra, with special focus on the transfer of properties to a ``spherical subalgebra". We also discuss general properties of Galois rings and algebras, which are a specialization of the previous notion, i.e., when we have an embedding in a skew group ring, and we obtain an important link between different approaches to it in the literature. Then we focus our study into finite multiplicative invariants on the ring of differential operators on the torus and fixed rings under the action of a finite group of algebra automorphisms of generalized Weyl algebras. We study freeness over the Harish-Chandra subalgebra and the Gelfand-Kirillov Conjecture  for them. Our last section constructs some concrete irreducible Harish-Chandra modules and show that they are holonomic. This paper also introduces the notion of an infinite rank generalized Weyl algebra.

\end{abstract}

\section{Introduction}

The idea of, given a pair algebra $U$ and subalgebra $\Gamma$, to study the induction and restriction functors between their categories of modules is very important. In Lie theory a very important case of this is the study of the so called \emph{Harish-Chandra modules} (cf. \cite[Chapter 9]{Dixmier}, \cite[Kapitel 6]{Jantzen}): $\mathfrak{g}$ is a finite dimensional Lie algebra and $\mathfrak{t}$ a reductive Lie subalgebra of it; $U=U(\mathfrak{g})$, $\Gamma=U(\mathfrak{l})$, and $M$ is a Harish-Chandra $\mathfrak{g}$-module for this pair if it is a finitely generated $U$-module whose restriction to $\Gamma$ is a direct sum of simple finite dimensional modules. In case $\mathfrak{g}$ is semisimple and $\mathfrak{l}=\mathfrak{h}$ is the Cartan subalgebra, we have in particular the theory of (generalized) weight modules, and the important BGG category $\mathcal{O}$ (see, e.g., \cite{Humphreys}).

Also in Lie theory, based in  the Gelfand-Tsetlin theory for $\mathfrak{gl}_n$ (cf. \cite{Zhelobenko}), another signficant pair was introduced in \cite{DFO0}, where $U=U(\mathfrak{gl}_n)$ and $\Gamma=\mathcal{Z}$ is the famous Gelfand-Tsetlin subalgebra, and finitely generated $U$-modules which decompose as a direct sum of generalized weight spaces with respect to $\mathcal{Z}$ the modules are the so called Gelfand-Tsetlin modules. Variations of the theme of Harish-Chandra modules are central to representation theory (cf. \cite{Mathieu}, \cite{Billig}, for instance).

A landmark in  generalized Lie theory was the introduction in \cite{DFO} of an abstract theory of Harish-Chandra algebras and modules, unifying classical Harish-Chandra modules \cite{Dixmier} \cite{Jantzen}, Gelfand-Tsetlin modules for $\mathfrak{gl}_n$ \cite{DFO0}, and generalized weight modules for generalized Weyl algebras \cite{Bavula} \cite{DGO}. In the case of $\mathfrak{gl}_n$, this is the largest category of modules for which there is a good understanding of simple objects (\cite{lots}) and which has lead to many developments in the past 30 years (cf. \cite{Proceedings}). An important recent contribution to this abstract theory is the work \cite{Fillmore}, which combines ideas from \cite{DFO} as well \cite{LM}, and has no assumption on the base field unlike classical cases where one usually assumes characteristic 0.


As a continuation of this theme, the theory of Galois algebras and orders was developed by V. Futorny and S. Ovsienko in \cite{FO} and \cite{FO2} is a refinement of these ideas, where it is exploited the embedding of these algebras in a skew group ring (a idea that goes back to the pioneering work of Block \cite{Block} in the classification of all the irreducible modules for $\mathfrak{sl}_2$ and the first Weyl algebra.). This theory has been successfully applied in the representation theory of many algebras, such as finite $W$-algebras of type $A$ \cite{FMO}, invariants of certain rings of differential operators \cite{FS0} \cite{FS3}, and invariants of certain quantum groups \cite{FS2}, OGZ-algebras, their q-analogues, $U_q(\mathfrak{gl}_n)$ and their parabolic versions \cite{Hartwig}, where the important notions of principal and rational Galois orders were introduced, and the alternating analogue of $U(\mathfrak{gl}_n)$ \cite{Jauch}. In \cite{Webster} an important variation of the theme, flag orders, were introduced, and a crucial link was established with certain objects arising for certain quantum field theories, namely Coulomb branches of $3 D \, \mathcal{N}=4$ gauge theories\footnote{Or, to be more precise, those objects and their 'flavour' deformations, the flavour space given by an algebraic torus \cite{BFN}.}, recently given a rigorous mathematical definition in \cite{BFN}, and their quantizations, called (following the terminology in \cite{lots}) spherical Coulomb branch algebras. More precisely, all of them are Galois orders \cite{Webster}. This development lead to breakthroughs in representation theory of many algebras, and these new techniques were able to settle longstanding conjectures, such as the freeness of the finite $W$-algebras of type $A$ over their Gelfand-Tsetlin subalgebras \cite{WWY} (which was conjectured in \cite{FMO} and proven there in some particular cases); at the same time, it shed new light into classical objects: for instance, the algebras in \cite{MvdB} are now known to be spherical Coulomb branch algebras of abelian gauge theories (cf. \cite{lots}). Further developments of the theory of Galois orders appear in \cite{FGRZ}, \cite{MV}, \cite{Jauch2}, \cite{Hartwig2}, \cite{Hartwig3} and \cite{Fillmore}. We remark that in this context, Harsish-Chandra modules are usually called Gelfand-Tsetlin modules (cf. \cite{FO2}, \cite{FGRZ}, \cite{lots}). We stick, however, with the terminology of \cite{DFO}.

This paper started as a natural continuation of \cite{FS0}, \cite{FS2}, \cite{FS3}, where using the realization of certain algebras as fixed rings of certain generalized Weyl algebras, many rings were shown to be Galois orders. However, the scope of this paper became far more general. Here is a description of its contents.

In the second section we make a detailed introduction of what are going to be the main objects of study in this work: Harish-Chandra algebras and modules, Galois rings and orders, and generalized Weyl algebras. We point out a connection with a previously studied class of algebras in ring theory: the so called FCR algebras \cite{KS}. We also include the relevant results in noncommutative invariant theory, and a detailed discussion of Morita theory, for it is going to play a very important role in what follows. Along the way, and in fact in the whole paper, we include the proof of many well known or folkloric results in the area that have never been published before or stated explictly in our generality (in particular, Proposition \ref{folklore}). 

The third section is a discussion about general Harish-Chandra algebras and modules, in the setting of \cite{DFO}\cite{Fillmore}, and contains the most important results of this paper. We begin by considering twisted generalized Weyl algebras and introduce infinite rank generalized Weyl algebras. Then we show new cases of quasicommutative and Harish-Chandra subalgebras, as well as pointing out that many interesting cases of modules in the literature lies in the framework of \cite{DFO}.

Then we show a series of results relating the category of Harish-Chandra modules with respect to $U, \Lambda$ and the one with respect to its ``spherical counterparts" $eUe, e\Lambda e$, where $e$ is an idempotent of $U$, specially in the important case where $U$ contains the group algebra $\k W$ of a finite group $W$ automorphisms of $U$ with $W(\Lambda) \subset \Lambda$, $|W|^{-1} \in \k$, and the idempotent is $e=1/|W| \sum_{w \in W} w$, with $\Lambda$ an affine and Noetherian algebra, a theme explored thoroughly in case of Galois orders and Hopf Galois orders in \cite{Webster} and \cite{Hartwig2}, respectively, and which is very useful in general (cf. \cite[Introduction]{EG}). We begin with a purely ring theoretical and somewhat surprising result that shows that quasicommutativity (\cite{DFO}) is a Morita invariant (Theore \ref{qc-is-Morita}). Then our main results are that if $\Lambda$ is quasicommutative, then so is $\Lambda^W$ (Theorem \ref{invariants-quasicommutative}), in the non-modular case, and we show the first natural examples since the pioneering work \cite{DFO} of a noncommutative Harish-Chandra subalgebra $\Gamma$ (Corollary \ref{mostrar}). The main result of this paper is that, under certain modest hypothesis, the categories of Harish-Chandra modules $\mathbb{H}(U,\Lambda)$ and $\mathbb{H}(eUe, \Lambda^W)$ are equivalent (Theorem \ref{Morita-HC}), generalizing results of \cite{Webster}. As purely abstract version of this resulting, mentioning only idempotents is also stated (Theorem \ref{abstract2}). We also show that if $U \supset \Lambda \supset \Gamma$ with $\Lambda, \Gamma$ quasicommutative and $\Lambda$ a finite $\Gamma$-module, then $\Lambda$ is a Harish-Chandra subalgebra if and only if $\Gamma$ is (Theorem \ref{remarkable}). 

Along the way, we show an interesting connection between a certain Harish-Chandra category and the (full) category $\mathcal{O}_{H_\mathfrak{c}}$ for the rational Cherednik algebras introduced in \cite{Berest} (see also \cite{Ginzburg2}; do not confuse with its 0th block, also called category $\mathcal{O}$ in \cite{GGOR} and many subsequent papers), in Theorem \ref{O-H-C}. 

Regarding Galois orders, we then apply our results to their theory in section four, again generalizing some results of \cite{Webster} (Theorem
\ref{non-principal-flags}). We give a concrete example that shows that not every pair algebra/Harish-Chandra subalgebra can be realized as a Galois ring. then, we show that if $\Gamma$ is a Notherian integrally closed Harish-Chandra subalgebra of $U$, then the frameworks for Galois orders in the original work of V. Futorny and S. Ovisenko \cite{FO}\cite{FO2} coincides with the framwork of J. T. Hartwig in \cite{Hartwig}. This is remarkable as the condition depends only on $\Gamma$, and it shows that if the obvious necessary conditions for an equivalence definitions hold, then indeed the equivalence holds.

Then we move to more specific topics. We discuss invariants under the action of a finite group of the ring of differential operators on the torus, considered previously in \cite{FO}, \cite{EFOS} and \cite{FS3}. We show that under a very general condition, they constitute a principal Galois order (Theorem \ref{multiplicative-Galois}). We also generalize the Gelfand-Kirillov Conjecture for them, discussed first in \cite[Theorem 3]{EFOS} and generalized in Theorem \ref{GK-1}. We point how that this particular result is related to a well described connection between Noether's problem and noncommutative Noether's problem (introduced in \cite{AD}), which was explored in \cite{FS} and \cite{SchwarzPan}, but is in some sense somewhat different in this case.

Then we move on to show that most results in \cite{FS3} can be generalized to the tensor product of a single rank 1 generalized Weyl algebra, when the group $G$ in the usual definition of Galois ring (cf. Definition \ref{Galois-ring}) is either (a) a permutation group (Theorem \ref{permutation}) or (b) a complex reflection group of the type $G(m,p,n)$ (Theorem \ref{main-objective}). In particular, we show that the invariants of the rank $n$ Weyl algebra $W_n(\k)^{G(m,p,n)}$ are rational Galois orders, the main unfinished goal of \cite{FS3}. This has been shown previously in \cite{LW}, in the case of groups $G(m,p,n)$ with a fairly sophisticated technique; our proof follows the technique of GWAs and is completely elementary. 
We also show that the Gelfand-Kirillov Hypothesis and its q-analogue holds for them (Theorem \ref{GK-2}).

In the fifth and last section we show that the usual modules in the category $\mathcal{O}$ for $W_n(\k)^G$, the fixed ring of the Weyl algebra $W_n(\k)$ under the action of complex reflection groups of type $G(m,p,n)$, are Harish-Chandra modules (this again in discussed implictly in \cite{LW}, but rely heavily on the theory of Coulomb branches \cite{BFN}. Our approach is, on the other hand, elementary). We also show that such modules are holonomic in an elementary way, simplifying some arguments in \cite{FS4}. One remarkable consequence is that modules in category $\mathcal{O}$ are Harish-Chandra with respect to two very different Harish-Chandra subalgebras ($\C[h^*]^W$ and $\Gamma=\C[t_1,\ldots,t_n]^W)$. We also adapt ideas from \cite{DFO} and \cite{DGO} to construct generic simple Harish-Chandra modules for certain invariants of generalized Weyl algebras under the action of a permutation group. A discussion of holonomic modules is also made, and is used in a small application constructing new holonomic modules for the ring of differential operators on the singular variety $\k^n/\mathcal{A}_n$.

\medskip

\textbf{Conventions} All modules will be left modules unless stated otherwise; we call an algebra Noetherian if it is both left and right Noetherian; the same with Artinian. $\otimes$, unadorned, means $\otimes_\k$, where $\k$ is the base field of the algebras under discussion.

\section{Preliminaries}

\subsection{Harish-Chandra subalgebras and Harish-Chandra modules I}

In this section the base field $\k$ can be completely arbitrary, following the results from \cite{Fillmore}, which generalize the foundational paper \cite{DFO}.

Let $U$ be an associative algebra and $\Gamma$ a subalgebra. We denote by $\cfs \, \Gamma$ its cofinite spectrum: that is, the set of maximal ideals $\m$ such that $\Gamma/\m$ is finite dimensional. If $\m$ is such an ideal, then $\Gamma/\m$ , by Wedderburn-Artin Theory, is a matrix ring over a finite dimensional division $\k$-algebra, and hence has a unique simple module $\mathsf{S}_\m$.

\begin{definition}\cite{DFO}\cite{Fillmore}
A finitely generated $\Gamma$-module $M$ is called a \emph{generalized weight module if} 

\[ M= \bigoplus_{\n \in \cfs \Gamma} M(\n), \]

where $M(\n)$ is the set of elements $x \in M$ such that $\n^{n(x)} x = 0$, for some $n(x) \in \Z_{>0}$. It is a \emph{uniform} generalized weight module if, for each $\n \in \cfs \Gamma$ such that $M(\n)\neq \{ 0 \}$, there is an $n \in \mathbb{Z}_{>0}$ such that $\n^n M(\n)=0$.\footnote{This is a reformulation of the notion introduced in \cite{Fillmore} of strong Harish-Chandra block module in case the equivalence relation is simply equality.}
\end{definition}

We call the support of $M$, $\operatorname{Supp} \, M$, the set of $\n \in \cfs \, \Gamma$ such that $M(\n) \neq 0$. In case all those generalized weight spaces are finite dimensional (such cases are discussed in \cite{DFO}, \cite{FO2}, \cite{Fillmore}), we call the Gelfand-Tsetlin multiplicity of $\n$ the value $\operatorname{dim}_\k \, M(\n)$.

In \cite{DFO}, two essential properties of $\Gamma$ were considered: being \emph{quasicentral} and \emph{quasicommutative}.

\begin{definition}\cite{DFO}
    $\Gamma$ is called quasicentral in $U$ if, for every $u \in U$, the $\Gamma$-bimodule $\Gamma u \Gamma$ is a finitely generated left and right $\Gamma$-module.
\end{definition}

\begin{definition}\cite{DFO}
$\Gamma$ is quasicommutative if, given distinct $\m, \n \in \cfs(\Gamma)$, $\operatorname{Ext}^1(\mathsf{S}_\m, \mathsf{S}_\n)=0$.
\end{definition}

Note that being quasicommutative is a property of $\Gamma$ alone, while being quasicentral is a property of the pair $(U, \Gamma)$.

The following proposition gives us equivalent and useful definitions of quasicommutativity.

\begin{proposition}\label{Fillmore}
$\Gamma$ is quasicommutative if and only if every finite dimensional $\Gamma$-module $M$ is a generalized weight module. In case $\Gamma$ is affine Noetherian, this is the same as to require that $\m \n = \n \m = \n \cap \m$ for $\n, \m \in \cfs \Gamma, \n \neq \m$.
\end{proposition}
\begin{proof}
    The first claim is \cite[Proposition 2.11]{Fillmore}. The second one \cite[Proposition 4]{DFO}.
\end{proof}

Of course the name deserves an explanation: we need to show that every commutative algebra is quasicommutative. The proof is simple but not a triviality, so we include it here for the sake of completeness.

\begin{proof}
    First note that $M$ is an $\Gamma/\operatorname{Ann}(M)$-module, the later which is a finite dimensional commutative algebra. Hence there is no loss in generality assuming $\Gamma$ finite dimensional. Hence, by standard commutative algebra, $\Gamma$ is Artinian and hence has a finite number of distinct maximal ideals $\m_1, \ldots, \m_s$ such that $(\m_1 \m_2 \ldots \m_s)^N=0$, for a $N>>0$. By the Chinese remainder theorem, $\Gamma \simeq \Gamma/\m_1^N \times \ldots \times \Gamma/\m_s^N$, and hence cleary $V$ decomposes as a direct sum of generalized weight spaces.
\end{proof}

The following Lemma is straightforward, having the same proof using linear algebra from the case of Lie algebras (cf. \cite{Dixmier}), and will be used without comment.

\begin{lemma}\label{basic-lemma}
    If a finitely generated $\Gamma$-module $M$ can be written as $M=\sum_{\n \in \cfs \, \Gamma} M(\n)$, then the sum is actually direct. All submodules and all factor modules of a generalized weight module are generalized weight modules.
\end{lemma}

We also remark that the notion of quasicommutative algebras is closely related to the notion of FCR algebras (\textbf{F}inite dimensional modules are \textbf{C}ompletely  \textbf{R}eductive), due to H. P. Kraft and L. W. Small \cite{KS}, and studied in \cite{KS2}, \cite{KSW}, \cite{FLS}, among others. These are algebras $A$ that have a) The reductive property: every finite dimensional module $V$ is completely reducible; and have b) The residually finite property: the intersection of the anihilators of all simple finite dimensional modules is 0. In fact we only need property a) to obtain quasicommutative algebras; for the lack of a better name, we will call algebras that satisfy a) above provisionally 'finite dimensional reductive' (FDR).

\begin{proposition}
    Every FDR algebra $A$ is quasicommutative.
\end{proposition}
\begin{proof}
    Let $V$ be a finite dimensional module. Then $V= \bigoplus_{i=1}^s V_i$, and each $V_i$ is an irreducible $A$-module. $\mathsf{Ann}(V_i)$ is a maximal ideal $\m_i$ of finite codimension in $A$, and $\m_i V_i=0$. Hence $V$ is a direct sum of generalized weight modules, and hence quasicommutative.
\end{proof}

\begin{example}\cite[Example 9]{DFO}
    Every enveloping algebra $U(\mathfrak{g})$ of a semisimple Lie algebra over a field of zero characteristic is an FCR algebra due to H. Weyl complete reducibility theorem and an important theorem of Harish-Chandra (\cite[2.5.7]{Dixmier}), and so is quasicommutative. The same holds if $\mathfrak{g}$ is reductive and $\k=\bar{\k}$ (cf. \cite[Corollary 1.7]{FLS}).
\end{example}

\begin{example}
    More examples of FCR - and hence, quasicommutative algebras - are $U_q(\mathfrak{g})$, for a semisimple Lie algebra $\mathfrak{g}$ over an algebraically closed field of zero characteristic when $q$ is a non-root of unity; certain generalized Weyl algebras such as down-up algebras; and the enveloping algbera of the complex finite dimensional Lie superalgebra $\mathfrak{osp}(1,2n)$, $n \geq 1$. Details for all these results can be found on \cite{KS2}. We also mention that invariants of the Weyl algebras under the action of tori are FCR \cite{MvdB}.
\end{example}

Not all quasicommutative algebra, however, is FCC (or FDR):

\begin{proposition}
    Let $\mathfrak{g}$ be a nilpotent Lie algebra over a field of zero characteristic. Then it is quasicommutative (cf. \cite[Example 9]{DFO}) but not necessarily FDR (and hence FCR).
\end{proposition}
\begin{proof}
By \cite[Lemma 2.1, Lemma 2.2]{Fillmore}. it suffices to show that, if $M, N$ are two finite dimensional irreducible modules for $\mathfrak{g}$, then $\operatorname{Ext}_{U(\mathfrak{g})}^1(M,N)=0$. As the Lie algebra is nilpotent this result is trivial since every irreducible finite dimensional module is the trivial module $\k$. On the other hand, if as our Lie algebra we consider $\mathfrak{n}$ the Lie algebra of all strictly upper triangular $n \times n$, the natural representation shows that $U(\mathfrak{n})$ is not FDR (and hence not FRC as well)
\end{proof}

The exact difference between FCR and quasicommutative algebras is that, in the later case it is required that $\operatorname{Ext}^1(S_\m, S_\n)=0$ for different elements of cofinite spectrum, while that in the case the FCR algebras, it is necessary that $\operatorname{Ext}^1(S_\m, S_\n)=0$ even if $\m=\n$ (compare Proposition \ref{Fillmore} with \cite[Theorem 3.1]{KSW}).

A simple but important remark, related to these matters, is the following

\begin{lemma}\label{maximal-free}
Let $A$ be an affine Noetherian algebra, $B$ another affine Noetherian algebra containing $A$ and such that $B$ is a free $A$-module. Then for each $\m \in \cfs \, A$ there exists a $\mathfrak{M} \in \cfs \, B$ with $\mathfrak{M} \cap A = \m$.
\end{lemma}
\begin{proof}
    Let $I$ be an index set such that, as $A$-module, $B= \sum_{i \in I} A_i$, where each $A_i \simeq A$. Choose $j \in I$. Let $\mathfrak{M}=\{ \sum_{i \in I} a_i| a_j \in \m \}$. It is not difficult to check that $\mathfrak{M}$ is the desired maximal ideal of $B$ with finite codimension.
\end{proof}



\begin{definition}\cite{DFO}
    Given a pair algebra $U$ and subalgebra $\Gamma$, then the latter is called a Harish-Chandra subalgebra if it is both quasicommutative and quasicentral. In this case, a generalized weight module with respect to $\Gamma$ is called a Harish-Chandra for the pair $(U, \Gamma)$. If the module is uniform generalized wight module, we call it uniform Harish-Chandra module.
\end{definition}

The class of Harish-Chandra modules is denoted by $\mathbb{H}(U, \Gamma)$, and it is an abelian extensionally closed full subcategory of $U-\operatorname{Mod}$. The full abelian subcategory of uniform Harish-Chandra modules is denoted $\mathfrak{H}(U, \Gamma)$.

The following proposition has been used in the case $\Gamma$ is commutative (e.g. \cite{FGRZ}, \cite{Hartwig}), but it is the first time that it is stated in this generality.

\begin{proposition}\label{folklore}
    Let $U$ be an associative algebra and $\Gamma$ a Noetherian Harish-Chandra subalgebra (not necessarily commutative). Given a finitely generated $U$-module $M$, it is Harish-Chandra if, and only if, $\Gamma$ acts locally finitely: that is, for each $v \in M$, $\Gamma.v$ is finite dimensional.
\end{proposition}
\begin{proof}
    Suppose $M$ is a Harish-Chandra module. Let $v \in M$. Then we can write $v=\sum v_{\m_i}$, where each $v_{\m_i}$ is anihilated by $\m_i^{n_i}$, $n_i \in \mathbb{Z}_{>0}$. Since $v$ is clearly annihilated by $\bigcap \m_i^{n_i}$, and by an easy induction argument we have that this ideal has cofinite dimension in $\Gamma$ (see next Lemma), we have that $\Gamma.v= \Gamma/(\bigcap \m_i^{n_i}).v$ is finite dimensional. Suppose now that for each $v \in M$, $\Gamma.v$ is finite dimensional. Then $\Gamma.v$ is a generalized weight module for $\Gamma$. So $M$ is a sum of generalized weight spaces, which, by Lemma \ref{basic-lemma}, must in fact be a direct sum.
\end{proof}

\begin{lemma}
    Let $A$ be a Noetherian algebra and $\m_1, \ldots, \m_s$ be different elements of $\cfs \, A$, with $n_i \in \mathbb{Z}_{>0}$, $i=1,...,s$. Then $\operatorname{dim} A/(\bigcap_{i=1}^s \m_i^{n_i}) < \infty$.
\end{lemma}
\begin{proof}
    Let $\m \in \cfs \, A$. It is enough to prove that, for each $N>0$, $\m^N$ has finite codimension; and that $M, M'$ are ideals of finite codimension, so is $M \cap M'$. As $A$ is Noetherian, for each $i>0$, $\m^i$ is a finitely generated ideal of $A$, and hence $\m^i/\m^{i+1}$ a finitely generated $A/\m$-module. Hence $\m^i/\m^{i+1}$ is finite dimensional. Assume $A/\m^N$ finite dimensional. We have a short exact sequence

    \[ 0 \rightarrow \m^N/m^{N+1} \rightarrow A/\m^{N+1} \rightarrow A/\m^N \rightarrow 0\]
    
    Since dimension is additive, $\m^{N+1}$ will have finite codimension as well.

    If $M, M'$ are ideals of finite codimension, then $A/(M \cap M') \into A/M \times A/N$. The codomain is a finite dimensional algebra;  hence so is its subalgebra $A/(M \cap M')$.
\end{proof}

\begin{example}
    If $\mathfrak{g}$ is a finite dimensional complex Lie algebra and $\mathfrak{l}$ is a reductive subalgebra on it, with $\k$ algebraically closed and 0 characateristic, then $U(\mathfrak{l})$ is a Harish-Chandra subalgebra of $U(\mathfrak{g})$, and this is way we recover the classical theory of Harish-Chandra modules, e.g., \cite[Chapter 9]{Dixmier} \cite[Kapitel 6]{Jantzen} (use the previous Proposition).
\end{example}

\begin{example}
    If $U=U(\mathfrak{gl}_n)$ and $\Gamma$ the Gelfand-Tsetlin subalgebra, then the Harish-Chandra modules are the so called Gelfand-Tsetlin modules, introduced in \cite{DFO0}\cite{DFO}. This was later generalized for $U(\mathfrak{so}_n)$ \cite{Mazorchuk} \cite{Disch} and to finite $W$-algebras of type $A$ \cite{FMO}.
\end{example}

\begin{example}
    If $\k=\C$ and $U=U_q(\mathfrak{gl}_n)$ with generic $q \in \C$ and $\Gamma=\mathcal{Z}$ the Gelfand-Tsetlin subalgebra \cite{MT2} \cite{Hartwig}. This later was generalized to $U'_q(\mathfrak{so}_n)$, the nonstandard q-deformed enveloping algebra of Gavrilik and Klimyk \cite{Disch}.
\end{example}

For details of classical Gelfand-Tsetlin theory see \cite{Zhelobenko} \cite{Molev}. Similarly, we have

\begin{example}
    If $\operatorname{char} \, \k=0$, $U=\k S_n$ and $\Gamma$ is the algebra of Jucys-Murphy elements, then $\Gamma$ is a Harish-Chandra subalgebra, and maximal commutative. $\operatorname{Specm} \, \Gamma$ parametrizes basis of irreducible $S_n$-modules \cite{OV}.
\end{example}

\begin{example}
    If $\operatorname{char} \, \k=0$ and $G$ is a finite group and $H$ a subgroup of $G$, then $\k H$ is a Harish-Chandra subalgebra of $\k G$.
\end{example}

We finish by introducing another definition from \cite{Fillmore}:

\begin{definition}
Let $U$ be an associative algebra and $\Gamma$ a Harish-Chandra subalgebra. We say that $\Gamma$ is a strong Harish-Chandra subalgebra if for each $\m \in \cfs \, \Gamma$, $U/ U \m$ and $U / \m U$ belongs to $\mathfrak{H}(U, \Gamma)$ (seen as left and right modules, respectively).  
\end{definition}


\subsection{A primer of invariant theory and Morita theory}


For use in the paper, we need to recall some results from noncommutative invariant theory and Morita theory. Our main references for the former are \cite{Passman2} and \cite{Montgomery}, and for the later \cite{Jacobson} and \cite{Lam2}.

First, we present a theorem that is a generalization of Noether's famous theorem in invariant theory\footnote{Namely, that if $A$ is an affine $\k$ algebra and $G$ a finite group of automorphisms of it, $A^G$ is also affine, and $A$ is a finitely generated $A^G$-module \cite{NoetherX}.} :

\begin{theorem}\label{MS}
Let $R$ be a finitely generated  Noetherian $\k$-algebra. If $G$ is a finite group of automorphisms of $R$, with $|G|^{-1} \in \k$, then $R^G$ is a finitely generated $\k$-algebra, $R^G$ is Noetherian, and $R$ is a finitely generated $R^G$-module.
\end{theorem}
\begin{proof}
    The first claim is proven in \cite{MS}. The second one in \cite[Corollary 1.12]{Montgomery} and the third onde in \cite[Corollary 5.9]{Montgomery}.
\end{proof}

In case our algebra is simple, more can be said:

\begin{theorem}\label{essential-invariant}
    Let $R$ be a simple algebra, and $G$ a finite group of outer automorphisms of $R$. Then
    \begin{enumerate}
        \item $R^G$ is a simple ring.
        \item $R$ is a finitely generated projective $R^G$-module
        \item $R^G$ and $R*G$ are Morita equivalent. In particular, if $R$ is Noetherian, $R^G$ is Noetherian.
    \end{enumerate}
\end{theorem}
\begin{proof}
    \cite[Theorem 2.4, Theorem 2.5, Corollary 2.6]{Montgomery}.
\end{proof}

In relation to last theorem, we recall a well known but extremely important fact:

\begin{proposition}
    Every automorphism of the Weyl algebra is outer.
\end{proposition}
\begin{proof}
    \cite[2.4.1]{Dumas}
\end{proof}

We also need a somewhat detailed discussion of the Lying Over property in the noncommutative setting.

\begin{theorem}\label{LO}
    Let $R$ be an affine algebra over $\k$, Noetherian. Let $G$ be a finite group of algebra automorphism of $R$ with $|G|^{-1} \in \k$. Then for each prime ideal $\mathfrak{p}$ of $R^G$ there is a prime ideal $P$ of $R$ \emph{lying over it}: $\mathfrak{p}=P \cap R^G$. Moreover, all other prime ideals $Q$ lying over are of the form $g.Q$, $g \in G$. In a more suggestive way, we have that the map

    \[ \operatorname{Spec} \, R \rightarrow \operatorname{Spec} \, R^G \]
is surjective and has finite fibers. We also have incomparability: if $P \subset P' $ are two ideals of $ \operatorname{Spec} \, R$ lying over the same prime ideal of $R^G$, then $P=P'$.
\end{theorem}
\begin{proof}
    \cite[Theorem 28.3]{Passman2}.
\end{proof}

Now we will discuss certain aspects of Morita theory important to us in great detail.

\begin{theorem} \label{Morita}
Let $R-\operatorname{Mod}$ and $S-\operatorname{Mod}$ be the categories of left modules of two Morita equivalent rings, and $(F,G)$ a pair of functors $F: R-\operatorname{Mod} \rightarrow S-\operatorname{Mod}$ and $G: S-\operatorname{Mod} \rightarrow R-\operatorname{Mod}$ that realize the equivalence of the categories. Then there is a bimodule $_
SP_{R}$ which is finitely generated projective as a $R$ or $S$ module, and a bimodule $_{R}Q_S$ which is finitely generated projective $S$ and $R$ module, such that $F$ is naturally isomorphic to $_SP \otimes_{R}-$ and $G$ to $_RQ \otimes_S-$. If $M$ belongs to $R-\operatorname{Mod}$, then $\GK \, M = \GK \, F(M)$; and similarly if it belongs to $S-\operatorname{Mod}$.
\end{theorem}
\begin{proof}
\cite[Morita I, Morita II]{Jacobson} and \cite[Theorem 1.7]{BH}
\end{proof}

Also an immediate consequence is that if $R-\operatorname{Mod}$ and $S-\operatorname{Mod}$ are equivalent, so are the categories of $\emph{right}$-modules

An important case is the following:

\begin{theorem}\label{Morita3}
Let $S$ be a ring and $e$ an element of it such that $e^2=e$ and $SeS=S$, i.e., $e$ is a full idempotent, and call $R=eSe$ (which has $e$ as its unit). Then the rings $S$ and $R$ are Morita equivalent, and the functors $F: S-\operatorname{Mod} \rightarrow R-\operatorname{Mod}$ and $G: R-\operatorname{Mod} \rightarrow S-\operatorname{Mod}$ that realize the Morita Equivalence between $S, R$ are explicitly given by $F(M)=eM$ and $G(N)=Se \otimes_{R} N$.
\end{theorem}
\begin{proof}
    \cite[Example 18.30]{Lam2}.
\end{proof}

As a  sub-case of previous Theorem, worth stating separately, is:

\begin{theorem}\label{Morita2}
Let $R$ be a $\k$-algebra which is simple and $G$ a finite group of outer automorphisms of it such that $|G|^{-1} \in \k$. Then calling $S=R*G$ and $e=1/G \sum_{g \in G} G$, we have $e^2=e$, $SeS=S$, $eSe \simeq R^G$, and hence the Morita equivalence between $R^G$ and $S$ is obtained as in the above Theorem; moreover we can write $F(M)$ as $M^G$.
\end{theorem}
\begin{proof}
    It is standard (see \cite{Montgomery}) that $e^2=e$, $eSe \simeq R^G$ and $eM=M^G$. By Theorem \ref{essential-invariant}, $S$ is a simple ring, and since $SeS$ is a non-null ideal of $S$, we must have $SeS=S$. By the previous theorem, we are done.
\end{proof}

\subsection{Galois rings and orders}

 Assume now we have a base field $\k$ which is algebraically closed and of zero characteristic. All fields, rings and algebras in this subsection are defined over $\k$.

The representation theory of Galois rings (orders) is a natural development of the ideas of last section, exploiting a common idea in representation theory, initiated in the work of Block \cite{Block} classifying the irreducible modules over $\mathfrak{sl}_2$ and the first Weyl algebra: embedd the algebras in question in skew group rings. The same idea has been successfully applied to generalized Weyl algebras (cf. \cite{Bavula}, \cite{BavulaK}).

To have an associative algebra as a Galois ring, we need the following data:

\begin{itemize}
    \item An associative algebra $U$ and a commutative subalgebra $\Gamma$ which is a domain such that $U$ is a finitely generated $\Gamma$-algebra and $\Gamma$ is Noetherian \footnote{this is not part of the original definition in \cite{FO}, but we are going to assume it in this paper since this condition holds in all known cases and simplifies significantly the exposition.}
    \item
    A finite Galois extension $L$ of $K:= \Frac \, \Gamma$; the Galois group will be denoted by $G$.
    \item 
    A submonoid $\mathcal{M}$ of $\Autk_\k \, L$ which is \emph{separating}: if $m, m' \in \M$ coincide in their restrictions to $K$, then $m=m'$.

    \item A $G$-action on $\M$ by conjugation.
\end{itemize}

Given this data, we can form the skew monoid ring $\L=L * \M$. $G$ acts on this ring by algebra automorphisms: $g.(lm)=g(l)m^g, \, g \in G, l \in L, m \in \M$, where we denote the action of $g$ in $m$ by $m^g$. Define $\K=\L^G$

\begin{definition}[\cite{FO}]\label{Galois-ring}
    
$U$ is a $\Gamma$-ring (or just Galois algebra) if $U \subset \mathcal{K}$ and $UK=KU=\mathcal{K}$

\end{definition}

\begin{lemma}\label{separating}
If $\mathcal{M} \cap G = id$ then $\mathcal{M}$ is separating. When $\mathcal{M}$ is a group, this is an equivalence.
\end{lemma}
\begin{proof}
    \cite[Lemma 2.2]{FO}.
\end{proof}

\begin{definition}
    Let $U$ be a Galois algebra in $\mathcal{K}$. Since every $u \in U$ belongs to $(L*\mathcal{M})^G$, we can write $u= \sum_{m \in \mathcal{M}} a_m m$, $a_m \in L$. Then we define $\operatorname{supp} \, u = \{m \in \mathcal{M}| a_m \neq 0 \}$.
    
\end{definition}

By definition, for every $u \in U$, $|\operatorname{supp} \, u|<\infty$.
\begin{proposition}\label{prop-supports}
    Suppose $U$ is a $\Gamma$-ring contained in $\mathcal{K}$ and is generated by a finite number of elements $u_1, \ldots, u_n$ as an algebra. If

    \[ \cup_{i=1}^n \operatorname{supp} \, u_i \]
    generates $\mathcal{M}$ as a monoid, them $U$ is a Galois $\Gamma$-ring.
\end{proposition}

\begin{proof}
    \cite[Proposition 4.1(1)]{FO}
\end{proof}




As an example we have the Weyl algebras $W_n(\k)$. They have the usual generators $x_1, \ldots, x_n, y_1, \ldots, y_n$ and relations $[x_i,x_j]=0$, $[y_i,y_j]=0$, $[y_i,x_j]=\delta_{ij}$, $i,j=1, \ldots, n$.

Call $t_i = y_i x_i$. Then we have a commutative polynomial subalgebra $\Gamma=\k[t_1,\ldots,t_n]$ inside $W_n(\k)$.

 In this paper, every time we consider the free abelian group $\mathbb{Z}^n$ we denote its canonical basis by $\varepsilon_1, \ldots, \varepsilon_n$. Despite the group being abelian, we will sometimes write $\varepsilon_i^{-1}$ instead of $- \varepsilon_i$, for notational convenience.


    

We have a natural map $\phi: W_n(\k) \rightarrow \Frac \, \Gamma * \Z^n$, with $\varepsilon_i(t_j)=t_j - \delta_{ij}, \, i,j=1,\ldots,n$, that sends $x_i \mapsto \varepsilon_i$ and $y_i \mapsto t_i \varepsilon_i^{-1}$. $\phi$ is clearly seen to be non-trivial; since $W_n(\k)$ is simple, the  map is an embedding. The support of $\phi(x_i), \phi(y_j), i,j=1,\ldots,n$ generates $\mathbb{Z}^n$ as a monoid. Since $\phi$ is an injection, we can recognize $W_n(\k)$ as a Galois $\Gamma$-ring inside $\k(t_1, \ldots, t_n)*\Z^n$.

To proceed with the theory, we need more notions.

\begin{definition}[\cite{FO}]
    Let $U$ be a $\Gamma$-ring in $\mathcal{K}$. Suppose that for every left finite dimensional vector space $W \subset \mathcal{K}$, $U \cap W$ is a finitely generated left $\Gamma$-module. Suppose also that for every right finite dimensional vector space $W \subset \mathcal{K}$, $U \cap W$ is a finitely generated right $\Gamma$-module. The $U$ is called a Galois order in $\mathcal{K}.$
\end{definition}

The next proposition is useful in finding Harish-Chandra subalgebras.

\begin{proposition}\label{order-HC}
    If $U$ is Galois order over $\Gamma$, and $\Gamma$ is a finitely generated $\k$ algebra, then $\Gamma$ is a Harish-Chandra subalgebra.
\end{proposition}

\begin{proof}
    \cite[Corollary 5.4]{FO}
\end{proof}

\begin{proposition}
    Suppose $U$ is a Galois algebra over $\Gamma$ in a skew monoid ring. If $U$ is a left and right projetive $\Gamma$-module, then $U$ is actually a Galois order.
\end{proposition}
\begin{proof}
    \cite[Corollary 5.2]{FO}
\end{proof}

An important related problem is a conjecture by V. Futorny that claims that if $U$ is a $\Gamma$-Galois order, then $U$ is necessarily a projective $\Gamma$-module.

\begin{definition}
    Let $U$ be a $\Gamma$-ring inside $\mathcal{K}$. Then we say that $U$ satisfies the Hartwig conditions if there is Noetherian integrally closed commutative domain $\Lambda$ with $G, \mathcal{M}$ subgroups of $\operatorname{Aut}_k \, \Lambda$ such that $L= \Frac \, \Lambda$, and the $G$ and $\mathcal{M}$ actions on $L$ are induced, via localization, from the action on $\Lambda$. We also require $\Gamma=\Lambda^G$.
\end{definition}

Galois ring that satisfy the Hartwig condtions are precisely the Galois rings under the slightly more restrictive point of view in \cite{Hartwig}. We did not have to prove that $\Lambda$ is a Noetherian $\Lambda^G$-module (\cite[Section 2.1, (A3)]{Hartwig}) precisely because we are assuming $\Lambda$ Noetherian (cf. Theorem \ref{MS}).

\begin{proposition}
    If $U$ is a $\Gamma$-ring that satisfy the Hartwig conditions, then $\Lambda$ is the integral closure of $\Gamma$ in $L$.
\end{proposition}
\begin{proof}
    \cite[Lemma 2.1]{Hartwig}
\end{proof}

Another essential concept introduced in \cite{Hartwig} is that of a principal Galois order.

\begin{definition}
    Let $U$ be a $\Gamma$-ring in $\mathcal{K}$. Every $u$ can be written uniquely in the form $u=\sum_{m \in \mathcal{M}} a_m m \, , a_m \in L$. Let $a \in L$. We define the evaluation $u(a)$ to be $\sum_{m \in \mathcal{M}} a_m m(a).$ 
\end{definition}

\begin{definition}
Let $U$ be a $\Gamma$-ring in $\mathcal{K}$ with the property that, for every $u \in U$, the evaluation $u(\gamma) \in \Gamma, \, \forall \gamma \in \Gamma$. Then $U$ is called a principal Galois order.    
\end{definition}

There is also a notion of co-principal Galois order:  there is an anti-isomorphism $\bar{\iota} : L*\M \rightarrow L*\M^{-1}, f*m \mapsto f*m^{-1}$, and we get that $U$ is Galois $\Gamma$-ring (order) in $(L*\M)^{G}$ if and only if $\bar{\iota}(U)$ is Galois $\Gamma$-ring (order) in $(L*\M^{-1})^G$. If $U$ is a Galois $\Gamma$-ring, then it is a co-principal Galois order if for every $X \in U$, $\bar{\iota}(X)(\gamma) \in \Gamma$, $\gamma \in \Gamma$.

As the name suggests, principal Galois orders are, indeed, Galois orders \cite[Theorem 2.21]{Hartwig}. All known Galois orders are actually principal Galois orders. A relevant open problem is to find an example of a Galois order that is not principal.

Finally, we have the notion of rational Galois order \cite[Definition 4.3]{Hartwig}

\begin{definition}
    Let $G < \operatorname{GL}_n(\k)$ be a finite complex reflection group on a finite dimensional vector space $V$, $\Lambda=S(V)$, $L=\Frac \, \Lambda$, $\Gamma=\Lambda^G$, $K=\Frac \, \Gamma$, $V$ acting on $L= \Frac \Lambda$ by translations: $w\cdot p(v)=p(v-w), p \in \Lambda, w,v \in V$. If $U$ is generated as an algebra by $\Gamma$ and elements $x_1,\ldots, x_r$, inside $L*V$ such that, for each $x_i$ there exists a semi-invariant $d_\chi$ of a character $\chi$ of $G$ such that $d_\chi x_i \in \Lambda*V$, by calling $\mathcal{M}$ the monoid generated inside $V$ by $\bigcup \operatorname{supp} x_i$, then $U$ is a $\Gamma$-rational Galois order in $(L*\mathcal{M})^G$
\end{definition}

Finite $W$-algebras of type $A$, OGZ algebras and $U(\mathfrak{gl}_n)$ are rational Galois orders (cf. \cite{Hartwig}). There is a similar notion of co-rational Galois order: that $\bar{\iota(U)}$ is raional Galois order in $(L*\M^{-1})$

\begin{theorem}
    Every (co-)rational Galois order is a (co-)principal Galois order.
\end{theorem}

\begin{proof}
    \cite[Theorem 4.2]{Hartwig}.
\end{proof}

The representation theory of such algebras was studied in detail in \cite{FGRZ} and \cite{Webster}.
\subsection{Generalized Weyl algebras}

In the begining of this section we deal with an arbitrary $D$ algebra over an arbitrary field $\k$.

Let's recall what a generalized Weyl algebra (henceforth, as usual, abreviated as GWA) is \cite{Bavula}. For a very detailed survey of many aspects of the theory of such algebras, we refer to:
\cite{Gaddis}.

\begin{definition}\label{def-GWA}
    Let $D$ be a ring, and $\sigma=(\sigma_1, \ldots, \sigma_n)$ a n-uple of commuting automorphisims: $\sigma_i \sigma_j = \sigma_j \sigma_i$, $i,j=1,\ldots,n$. Let $a=(a_1,\ldots,a_n)$ be a n-uple of regular elements belonging to the center of $D$, such that $\sigma_i(a_j)=a_j, j \neq i$. The generalized Weyl algebra $D(a, \sigma)$ of rank $n$ with base ring $D$ is generated over $D$ by $X_i^+, X_i^-$, $i=1,\ldots, n$ and relations

    \[ X_i^+d= \sigma_i(d) X_i^+ \, ,  X_i^- d= \sigma_i^{-1}(d) X_i^- \, , d \in D, \]

    \[ [X_i^+, X_j^+]=[X_i^-, X_j^-]=[X_i^+, X_j^-]=0, \, i,j=1,\dots,n, i\neq j, \]

    \[ X_i^- X_i^+ = a_i, \, \, X_i^+  X_i^- = \sigma_i(a_i). \]
\end{definition}

$D(a, \sigma)$ is a free left and right $D$-module. If $D$ is a domain, $D(a, \sigma)$ is a domain. If $D$ is left or right Noetherian, so is $D(a, \sigma)$ \cite{Bavula}.


\begin{proposition}\label{prop-basic-GWA}
    Let $D(a, \sigma)$ be a GWA of rank one. Then $D(a, \sigma)^{\otimes \, n}$ is a generalized Weyl algebra of rank $n$ $D'(a', \sigma')$:

     $D'= D^{\otimes \, n}$. 

     $a_i= 1 \otimes \ldots \otimes 1 \otimes a \otimes 1 \otimes \ldots$, with $a$ in the $i$-th position.

     $\sigma_i = id \otimes \ldots \otimes id \otimes  \sigma \otimes id \otimes  \ldots$, with $\sigma$ in the $i$-th position.
\end{proposition}
\begin{proof}
    \cite{Bavula}
\end{proof}

\begin{example}

$W_1(\k)$ is the GWA $D(a, \sigma)$, where $D=\k[h]$, $a=h$, $\sigma(h)=h-1$. $W_1(\k)^{\otimes \, n} = W_n(\k)$.

\end{example}

We have the following improvement of \cite[Theorem 14]{FS3}

\begin{theorem}\label{theorem-GWA}
    Let $D$ be a commutative Noetherian domain over an algebraically closed field of zero characteristic, $D(a, \sigma)$ a rank $n$ GWA. Consider the natural map $\phi: \mathbb{Z}^n \rightarrow G$, where $G$ is the subgroup of $\operatorname{Aut}_\k \, D$ generated by $\sigma_1, \ldots \sigma_n$, given by $\sum_{i=1}^n z_i \varepsilon_i \mapsto \sigma_1^{z_1} \ldots \sigma_n^{z_n}, \, z_i \in \mathbb{Z}$. $\phi$ is clearly surjective. If $\phi$ is also injective (and hence, a bijection), then the map

    \[ X_i^+ \mapsto \varepsilon_i \, \, X_i^- \mapsto a_i \varepsilon_i^{-1} \]

    from $D(a, \sigma)$ to $\Frac \, D * \mathbb{Z}^n$, where $\varepsilon_i=\sigma_i$, is a realization of the GWA as a principal Galois order.
\end{theorem}
\begin{proof}
    By the above mentioned Theorem of \cite{FS3}, the above map realizes $D(a,\sigma)$ as a $\Gamma$-ring in $F*\Z^n$, where $F= \Frac \, D$ and $\Gamma = D$. The fact that is a principal Galois order is clear because in the embedding of the algebra in $F*\Z^n$ there are no denominators in the formulas.
\end{proof}

If the conditions of the above Theorem are not satisfied, we still have something to say (cf. Proposition \ref{TGWA-HC}).


The following definition is going to be used frequently in the subsequent work.

\begin{definition}\label{tensorial}
We call a rank $n$ GWA \emph{tensorial} if it is the tensor product of $n$ times the same generalized Weyl algebra of rank 1 $D(a, \sigma)$ with infinite order automorphism. We denote the GWA $D(a, \sigma)^{\otimes n}$ by $D(a,\sigma)^n$.
\end{definition}

It will also be important to have a simplicity criteria for GWA`s. We have (\cite[Theorem 4.5]{BF}):

\begin{theorem}\label{GWA-simplicity-1}
    Let $D(a,\sigma)$ be a rank $n$ GWA. Then it is simple if and only if the following holds:
    \begin{enumerate}
        \item $D$ does not have any proper ideal stabilized by all $\sigma_i, i=1,\ldots,n$.
        \item The subgroup of the factor group $\operatorname{Aut}_\k \, D / \operatorname{Inn}_\k \, D$, where $\operatorname{Inn}_\k \, D$ is the subgroup of inner autormophisms, generated by the images of $\sigma_1, \ldots, \sigma_n$, is isomorphic to the free abelian group $\Z^n$
        \item $D a_i + D \sigma_i^m(a_i)=D, m \leq 0, i=1,\dots, n$.
    \end{enumerate}
\end{theorem}

\begin{corollary}\label{GWA-simplicity-2}
    \begin{itemize}
        \item Call a rank 1 GWA $A$ of classical type if $A=\k[h](a,\sigma)$, $\sigma(h)=h-1$. If there is no irreducible polynomial $p[h]$ such that both $p, \sigma^i p$ are multiples of $a$, for any $i \geq 0$, the GWA of classical type is simple.
        \item Call a rank 1 GWA $A$ of quantum type if $A=\k[h^{\pm 1}](a,\sigma)$, $\sigma(h) = \lambda h$, $\lambda \neq 0$, $\lambda \neq 1$, $\lambda$ is not a root of unity. If there is no irreducible polynomial $p[h]$ with $p, \sigma^i(p)$ both being multiples of $a$, $i \geq 0$, the GWA of quantum type is simple.
        \item If a generalized Weyl algebra $A$ is the tensor product $\bigotimes_{i=1}^n A_i$, where each $A_i$ is a simple GWA's of either classical or quantum type, then $A$ is also simple

    \end{itemize}
\end{corollary}
\begin{proof}
    Everything follows from the previous Theorem and \cite[Corollary 4.8]{BF}.
\end{proof}

We state now a simple fact that will be important later. We include the proof for the sake of clarity

\begin{proposition}\label{skew-is-GWA}
    Let $D$ be a commutative algebra and $A=D*\mathbb{Z}^n$ a skew product ring. Then $A$ is the GWA with base ring $D$, $\sigma_i=\varepsilon_i, i=1,\ldots,n$ and each $a_i=1, i=1,\ldots,n$.
\end{proposition}
\begin{proof}
    Consider the map $\phi:D*\mathbb{Z}^n \rightarrow D(a,\sigma)$ that sends $\varepsilon_i$ to $X_i^+$ and $-\varepsilon_i$ to $X_i^-$, and is the identity on $D$. This map is clearly an homomorphism. It sends a base as a free $D$-module in the domain to another one in the codomain. Hence it is an isomorphism.
\end{proof}

We finish with an important theorem in the invariant theory of GWAs.

\begin{theorem}[\cite{JW}]\label{JW}
    Let $D(a, \sigma)$ be a rank one GWA. We can define an action of the clyclic group $G_m$ on it as follows: picking a primitive $m$-th root of unity $\zeta$, and a generator $g$ of the group, we can make it act on the GWA by $g.X^+ = \zeta X^+$. $g X^-=\zeta^{-1} X^-$, and $g(d)=d, \, d \in D$. This really makes $G_m$ acts by algebra automorphisms of $D(a, \sigma)$. We have an isomorphism
    $\phi: D(a, \sigma)^{G_m} \simeq D(a^*, \sigma^m)$, where

\[ a^* = a \sigma^{-1}(a) \ldots \sigma^{-(m-1)}(a) .\]

Calling $Y^+$ and $Y^-$ the GWA generators of $ D(a^*, \sigma^m)$, the isomorphism $\phi$ can be explicitly given by $(X^+)^m \mapsto Y^+$, $(X^-)^m \mapsto Y^-$, $d \mapsto d, d \in D$.
\end{theorem}

\begin{remark}
    The author stresses that, despite the result sounding obvious, its proof is far from trivial.
\end{remark}

This, of course, is a generalization of the case where $D(a,\sigma)$ is the first Weyl algebra, considered in \cite[pp.225]{FS3}, due to Bavula \cite{Bavula}. It is going to play an essential role in what follows.

\section{Harish-Chandra subalgebras and Harish-Chandra modules II}

In this section, we will work on the setting of Harish-Chandra modules of \cite{DFO} \cite{Fillmore}. The first subsection deals with generalized and twisted generalized Weyl algebras, and a generalization of the former to an infinite rank case.

The second subsection is a recollection of many instances in the literature where pairs of algebra/Harish-Chandra subalgebra appears in the literature.

The third subsection is somewhat ring theoretical in flavour, and discuss quasicommutative algebras their invariants, and the essential fact that quasicommutativity is a Morita invariant; and then moves to the main results of this paper: the comparison of the Harish-Chandra modules of an algebra and its ``spherical subalgebra".

\subsection{TGWAs and infinite rank generalized Weyl algebras}

We begin with a simple lemma, whose proof we leave as a simple exercise.

\begin{lemma}\label{lemma-HC}
Let $U$ be an algebra by elements $\{ u_i \}_{i \in I}$ such that $Du_i=u_iD$, where $D$ is a subalgebra, for every $i \in I $. Then $D$ is quasicentral.

\end{lemma}

Now we apply this lemma. The following result is probably known (see, e.g, \cite{DGO}), but the author could not find an explicit statement of it on the literature.

\begin{proposition}\label{TGWA-HC}
    Let $A$ be either a GWA $D(a,\sigma)$, or a TGWA\footnote{abreviation for twisted generalized Weyl algebra, introduced in \cite{MT}. For a survey on the subject, see, e.g., \cite{Gaddis}.}$\mathcal{A}_\mu(D,\sigma,t)$ (in the notation of \cite{FHTGWA}). Then $D$ is quasicentral subalgebra. If it is quasicommutative, then it is a Harish-Chandra subalgebra. 
\end{proposition}
\begin{proof}
In both cases (cf. \cite{Bavula}, \cite{FHTGWA}), the algebra is generated by elements $X_i^+, X_i^-$ that satisfy $DX_i^\pm=X_i^{\pm}D$.
\end{proof}

We are now going to introduce the definition of an infinite rank GWA. \footnote{The idea of such algebras is alluded in \cite{FGM}, but it was not explored.}

For the moment, $D$ will be just a ring.

\begin{definition}
    Let $D$ be a ring and $\mathbb{I}$ an indexing set of any infinite cardinality $\aleph_\iota$. Let $\{ a_i \}_{i \in \mathbb{I}}$ be a set of regular elements on the center of $D$ and $\{ \sigma_i \}_{i \in \mathbb{I}}$ be a set of commuting automorphisms of $D$ such that $\sigma_i(a_j)=a_j, \,, i \neq j$. A generalized Weyl algebra of rank $\aleph_\iota$ is an algebra generated by $D$ and a set of symbols $X_i^+, X_i^-, i \in \mathbb{I}$, subject to the relations

    \[ X_i^+d=\sigma_i(d) X_i^+, X_i^-d=\sigma_i^{-1}(d) X_i^{-}, \]

    \[ [X_i^+, X_j^+]=[X_i^-,X_j^-]=[X_i^+, X_j^-]=0, i \neq j\]

    \[ X_i^- X_i^+ =a_i, \, X_i^+ X_i^-=\sigma(a_i), \, i, j \in \mathbb{I}, d \in D.\].

We denote such algebra by $D(a, \sigma)$, just like the ordinary case.    
\end{definition}

A particular case is the infinite rank Weyl algebra $W_\infty(\k)$, whose representation theory was considered in \cite{FGM}, \cite{BBF}, among others. In this case $D=\k[h_i]_{i \in \mathbb{I}}$, $\sigma_i(h_i)=h_i-1$, and $\sigma_i$ fixes $h_j$ if $j \neq i$; and each $a_i=h_i$.

For the next result, assume $D$ to be an algebra over any field.

\begin{theorem}
    Let $D(a, \sigma)$ be an infnite rank GWA. It is a free left and right $D$-module. If $D$ is a domain, $D(a, \sigma)$ is a domain. $D$ is a quasicentral subalgebra of $D(a,\sigma)$. Hence, if $D$ is also quasicommutative, then $D$ is a Harish-Chandra subalgebra.
\end{theorem}
\begin{proof}
    Let $\mathfrak{Z}=\bigoplus_{i \in \mathbb{I}} \mathbb{Z}$, with basis $\{ \varepsilon_i \}_{i \in \mathbb{I}}$. There is a natural notion of degree of a monomial on the $X_i^\pm$ - this degree belongs to $\mathfrak{Z}$. Let $\mathbb{I}_{fin}$ be the collection of finite subsets of $\mathbb{I}$. Then $D(a,\sigma)=\sum_{\mathsf{I} \in \mathbb{I}_{fin}} D^\mathsf{I}$, where $D^\mathsf{I}$ is the the set of $D$-linear combinations of monomials with degree $\mathsf{I}$. Hence $D(a,\sigma)$ is a free $D$-module. Given $u, v \in D$ two nonzero elements, all monomials appearing in $u, v$ involves indices that belongs to a certain $\mathsf{I} \in \mathbb{I}_{fin}$. Hence we can imagine $u,v$ as living in the obvious rank $|\mathsf{I}|$ GWA $D(a|_\mathsf{I},\sigma|_\mathsf{I})$ - and hence the their product is nonzero by the well known finite rank theory. Finally, $D$ is quasicentral by Lemma \ref{lemma-HC}.
\end{proof}

Suppose finally that $D$ is a commutative domain. There is a natural epimorphism from $\mathfrak{Z}$ into the subgroup of $\operatorname{Aut}_\k \, D$ generated by $\{ \sigma_i \}_{i \in \mathbb{I}}$. If this epimorphism is actually an isomorphism (cf. Theorem \ref{theorem-GWA}), we have:

\begin{theorem}
Let $D(a,\sigma)$ be a rank $\aleph_i$ generalized Weyl algebra. Then it is a principal Galois order in $F*\mathfrak{Z}$\footnote{This is the first instance in the literature where $\M$ is not $\mathbb{Z}^n$, $\mathbb{N}^n$, or the semi-direct product of one of them with a finite group.}, where $F=\Frac \, D$.
\end{theorem}

So in particular $W_\infty(\k)$ is a Galois order. We expect the above Theorem to have interesting implications to the representation theory.

\subsection{Harish-Chandra subalgebras and modules in the literature}

In this section we are going to discuss some important examples from representation theory, not previously explicitly noted before, where we have Harish-Chandra subalgebras and Harish-Chandra modules.

The following lemma will be very useful in verifying that some algebras are quasicentral.

\begin{lemma}\label{Berest}
    Let $A$ be an affine commutative algebra and $M$ an $A$-bimodule such that the adjoint action of $A$ is locally nilpotent. Then for every $v \in M$, $AvA$ is a finitely generated left and right $A$-module.
\end{lemma}
\begin{proof}
\cite[Lemma 2.8]{Berest}
\end{proof}

This is the case, for instance, when $\operatorname{ad} a$ is a semisimple operator for each $a \in A$.

As a first appplication, we have

\begin{proposition}
    Let $X$ be an irreducible affine variety, not necessarily smooth, and $\mathcal{D}(X)$ its ring of differential operators. Then $\mathcal{O}(X)$ is a Harish-Chandra subalgebra of $\mathcal{D}(X)$. Similarly, if $G$ is a finite group of automorphisms of $X$, and $X$ is smooth, $\mathcal{O}(X)^G$ is a Harish-Chandra subalgebra of $\mathcal{D}(X)^G$.
\end{proposition}

\begin{proof}
    The first claim follows from the Grothendieck definition of ring of differential operator (cf. \cite[Chapter 15]{McConnell}) and Lemma \ref{Berest}. In the second situation, we have that $\mathcal{D}(X)^G$ is a subring of $\mathcal{D}(X/G)$ by \cite[Theorem 5]{Levasseur}. $\mathcal{O}(X/G)$ acts ad-nilpotently in the larger ring, and hence also on the smaller. Hence we conclude as before.
\end{proof}

The case of the Weyl and their invariant subalgebras is particularly interesting. We have $W_n(\k) \simeq \mathcal{D}(\mathbb{A}^n)$. In the classical work of Galois orders \cite{FO} \cite{FO2}, the Harish-Chandra subalgebra choosen is $\Gamma=\k[t_1, \ldots, t_n]$, for this is the algebra which makes more sense in the context of generalized Weyl algebras (cf. \cite{DGO} \cite{BBF} and references therein). $\Gamma$ is also maximal commutative\footnote{For a proof, combine Proposition \ref{diff-ops} and \cite[Proposition 2.14]{Hartwig}.}. On the other hand, we can choose another maximal commutative Harish-Chandra subalgebra, $\k[\mathbb{A}^n]$; which, moreover, is ad-nilpotent. 

In sections bellow, we are going to see that for many finite $G < \operatorname{Aut}_\k \, W_n(\k)$, $\Gamma^G$ is a Harish-Chandra subalgebra of $W_n(\k)^G$ and we can realize this pair as a Galois order.

Using the theory of Hopf Galois orders by J. Hartwig \cite{Hartwig2}, Harish-Chandra modules for $\mathcal{D}(\mathbb{A}^n)^G$ and Harish-Chandra subalgebra $\k[\mathbb{A}^n]^G$ can also be studied (cf. \cite[Section 5]{Hartwig2}). We will leave this investigation for a future opportunity.

Harish-Chandra subalgebras also appear in the context of infinite dimensional Lie algebras. For example, in the work \cite{Billig}  the irreducible weight modules with finitely dimensional weight spaces were considered for the algebra $\mathcal{W}_n(\k)$ of vector fields on the algebraic torus $\k^{\times n}$. Namely, these are the $\mathcal{W}_n(\k)$-modules which decompose as a direct sum of finite dimensional eigenspaces for the Cartan subalgebra $\mathfrak{h}=\sum \k t_i d_i$, where $d_i(t_i)=\delta_{ij}$ is a derivation. As $\mathfrak{h}$ acts diagonally on $\mathcal{W}_n(\k)$, $U(\mathfrak{h})$ acts locally ad-nilpotently in $U(\mathcal{W}_n)$. Hence, by Lemma \ref{Berest}, it is a Harish-Chandra subalgebra, and the modules studied in \cite{Billig} are Harish-Chandra modules for the pair $U(\mathcal{W}_n(\k))$, $U(\mathfrak{h})$. We have the same situation in the case our algebra is the enveloping algebra of the Virasoro algebra \cite{Mathieu}\footnote{It is relevant to remark that the enveloping algebra of these infinite dimensional Lie algebras are not left or right Noetherian \cite{SW} . } .

In the quantum group case, we have the Drinfeld-Jimbo presentation of $U_q(\mathfrak{g})$ (see, e.g, \cite[I.6.3]{Brown}), where $\mathfrak{g}$ is a semisimple Lie algebra over an algebraically closed field of zero characteristic and $q \in \k$ is different from 0 and 1. By the PBW-Theorem (\cite[I.6.8]{Brown}), $\Gamma=\k[K_\alpha]_{\alpha \in \Sigma}$, where $\Sigma$ is the root system of $\mathfrak{g}$, is a Harish-Chandra subalgebra of $U_q(\mathfrak{g})$, and weight modules are a particular example of Harish-Chandra modules.

The study of more general representations  of Lie algebras of vector fields on smooth affine varieties started in \cite{Billig2}. They are very interesting algebras, being simple and infinite dimensional. However, since unlike $\mathcal{W}_n(\k)$, in general they may not have any semisimple or nilpotent element, the so called subclass of $\mathcal{A}\mathcal{V}$-modules, introduced in \cite{Billig} and throughly studied in \cite{Billig2} and subsequent papers, became an important focus of attention. These are $\mathcal{V}$-modules that have a compatible action of the algebra of polynomial functions on the smooth affine variety, denoted by $A$. Equivalently, they are representations of the smash product algebra $\mathcal{H}(\mathcal{V}):=A*U(\mathcal{V})$ (where $\mathcal{V}$ acts on $A$ by derivations as usual). Similarly to the case of algebras of differential operators, $A$ is subalgebra of $\mathcal{H}(\mathcal{V})$ whose action is ad-nilpotent. Hence, again by Lemma \ref{Berest}:

\begin{proposition}
    $A$ is a Harish-Chandra subalgebra of $\mathcal{H}(\mathcal{V})$.
\end{proposition}

\begin{lemma}
    Let $A$ be an affine commutative algebra over an algebraically closed field and $P$ a finitely generated projective module. Then $P$ is \emph{not} a generalized weight space module for $A$.
\end{lemma}
\begin{proof}
    It is canonical that there exists another finitely generated projective module $Q$ such that $P \oplus Q$ is a free module of rank $n$. Hence the action of $A$ on $P$ must be torsion free.
\end{proof}

Most research on $\mathcal{A}\mathcal{V}$-modules has focused on the case the modules are finitely generated over $A$, and in this case, by \cite[Proposition 1.2]{BR}, those modules are finitely generated projective. Since, without any restrictions, it is too difficult to expect a complete understanding of $\mathcal{V}$-modules, perhaps the direction of studying Harish-Chandra modules in the category $\mathbb{H}(\mathcal{H}(\mathcal{V}),A)$ may be a promising one.



\subsection{Quasicommutative algebras and ``spherical subalgebras"}

The first part of this section is dedicated to the study of the ring theoretical notion of being a quasicommutative algebra. The main results of this section is that it is a Morita invariant (Theorem \ref{qc-is-Morita}) and behaves well when we pass to fixed rings under the action of a finite group whose order is invertible in $\k$ (Theorem \ref{invariants-quasicommutative}). Then we move to the theme of passing from the representation theory of an algebra $U$ to a spherical subalgebra $eUe$, where $e$ is an idempotent; this is a atheme considered in \cite{Webster} and \cite{Hartwig2} (see also \cite{EG}) -  but considered here in the much more general setting of \cite{DFO} and \cite{Fillmore}.

We begin, first, with a simple Proposition (compare with \cite[Corollary 1.7]{FLS}).

\begin{proposition}
If $\Gamma_1$ and $\Gamma_2$ are quasicommutative algebras, so is $\Xi:=\Gamma_1 \otimes \Gamma_2$. If $\Gamma_i$ is a Harish-Chandra subalgebra in $U_i$, $i=1,2$, then $\Gamma_1 \otimes \Gamma_2$ is a Harish-Chandra subalgebra of $U_1 \otimes U_2$.
\end{proposition}
\begin{proof}
    Let $V$ be a finite dimensional $\Xi$-module. We can identify $\Gamma_i, i=1,2$ as subalgebras of $\Xi$ that centralize each other: $\Gamma_1 \Gamma_2 = \Gamma_2 \Gamma_1$. As $\Gamma_i, i=1,2$ are quasicommutative, there are subsets $S_1, S_2$ of $\cfs \, \Gamma_1$ and $\cfs \, \Gamma_2$. respectively, such that $V$ restricted to $\Gamma_i$ is a generalized weight module with $S_i$ as its support, $i=1,2$. $V|_{\Gamma_1} = \bigoplus_{\m \in S_1} V(\m)$, and since $\Gamma_1 \Gamma_2=\Gamma_2 \Gamma_1$, each $V(\m)$ is also a $\Gamma_2$-module. Then, clearly, as a $\Xi$-module, $V=\bigoplus_{(\m,\n) \in S_1 \times S_2} V( \, \langle \m, \n \rangle \,)$, where $\langle \m, \n \rangle$ is the element $\m \otimes \Gamma_2 + \Gamma_1 \otimes \n$ of $\cfs \, \Xi$, where $\m \in S_1, \n \in S_2$. That $\Xi$ is quasicentral in $U_1 \times U_2$ is a simple exercise.
\end{proof}

Our second result follows the philosophy that PI-algebras behave like commutative algebras in many cases (cf. \cite{Drensky}).

\begin{theorem}
If $\Lambda$ is a prime affine Azumaya algbera\footnote{cf. \cite[13.7.6]{McConnell}.} over an algebraically closed field, then it is quasicommutative.    
\end{theorem}
\begin{proof}
    A standard argument (e.g., \cite[I.13.5]{Brown}) shows that, if $V$ is any irreducible $\Lambda$-module, $\Lambda/\operatorname{Ann}(V) \simeq M_t(\k)$, for a certain $t$ between $1$ and the PI-degree of $\Lambda$. Since $\Lambda$ is Azumaya, for all irreducible $V$, $t=\operatorname{PI-deg} \, \Lambda$, which we will call $n$ (cf. \cite[13.7.14]{McConnell}). So every irreducible module $V$ of $\Lambda$ has a natural structure of $M_n(\k)$-module. Since $M_n(\k)$ is a semisimple ring, given two irreducible modules $V, W$ for $\Lambda$, $\operatorname{Ext}^1(V,W)=0$. Hence $\Lambda$ is quasicommutative.
\end{proof}

\begin{example}
Examples of prime affine Azumaya algebras are abundant: if $\k$ is algebraically closed with prime characteristic, the Weyl algebra, or more generally, any ring of crystalline differential operators on an smooth affine variety, is of this class (e.f. \cite{BMR}); the Jordan simple localization of the Maltsionits q-deformation of the Weyl algebra at roots of unity is another example, or $\mathcal{O}_q(\k^{\times n})$ at roots of unity (cf. \cite{Brown} and \cite{BY}).
\end{example}

We are going to show the very important result that quasicommutativity is a Morita invariant.

\begin{lemma}
    Let $\Lambda$ and $\Gamma$ be two algebras Morita equivalent. Then the sets $\cfs \, \Lambda$ and $\cfs \, \Gamma$ are in bijection. More precisely, there is a $(\Lambda, \Gamma)$-bimodule $P$, which is a finitely generated projective left $\Lambda$-module/right $\Gamma$-module, such that to each $\m \in \cfs \, \Lambda$ corresponds a unique $\mu(\m) \in \cfs \, \Gamma$ such that $\m P = P \mu(\m)$, and vice-versa.
\end{lemma}
\begin{proof}
    By standard Morita theory (\cite[Morita I (6)]{Jacobson}), the lattice of ideals of both algebras is isomorphic. Even more, by \cite[18.47, 18.49]{Lam2}, the lattice isomorphism reduces to a bijection of the cofinite spectrum which, again by \cite[Morita I (6)]{Jacobson}, is of the form described.
\end{proof}

\begin{theorem}\label{qc-is-Morita}
    Let $\Lambda$ and $\Gamma$ be Morita equivalent algebras. Then if $\Lambda$ is quasicommutative, then so is $\Gamma$: quasicommutativity is a Morita invariant.
\end{theorem}
\begin{proof}
    In the notation of the previous Lemma, and using Theorem \ref{Morita}, let $Q$ be the ($\Gamma, \Lambda$)-bimodule such that the functor $G(M):$ $M \mapsto Q \otimes_\Lambda M$ realizes the equivalence between $\Lambda-\operatorname{Mod}$ and $\Gamma-\operatorname{Mod}$. Let $\m \in \cfs \, \Lambda$ corresponds to $\mu(\m) \in \cfs \, \Gamma$. Then by \cite[18.49]{Lam2}, $G(S_{\m})$ corresponds to $S_{\mu(\m)}$. In $\Lambda-\operatorname{Mod}$ we have a free resolution $\ldots \rightarrow \Lambda^{n_2} \rightarrow \Lambda^{n_1} \rightarrow S_\m \rightarrow 0$. Let $\n \in \cfs \, \Lambda, \n \neq \m$. Apply the functor $\operatorname{Hom}_\Lambda(S_\n, \cdot)$. We have $\ldots \rightarrow \operatorname{Hom}_\Lambda(S_\n, \Lambda^{n_2}) \rightarrow \operatorname{Hom}_\Lambda(S_\n, \Lambda^{n_1}) \rightarrow \operatorname{Hom}_\Lambda(S_\n, S_\m) \rightarrow 0$. $\operatorname{Ext}_\Lambda^1(S_\n, S_\m)$ is the first homology group of this complex. Apply $G(\cdot)$ to this chain complex to obtain $\ldots \rightarrow \operatorname{Hom}_\Gamma(S_{\mu(n)}, \Gamma^{n_2}) \rightarrow \operatorname{Hom}_\Gamma(S_{\mu(n)}, \Gamma^{n_1}) \rightarrow \operatorname{Hom}_\Gamma(S_{\mu(\n)}, S_{\mu(\m)}) \rightarrow 0$ (where $G(\Lambda^n)=\Gamma^n$ again follows from \cite[18.49]{Lam2}). $\operatorname{Ext}_\Gamma^1(S_{\mu(\n)},S_{\mu(\m)})$ is the first homology group of this new complex, and since $G(\cdot)$ is an equivalence of categories, $\operatorname{Ext}_\Lambda^1(S_\n, S_\m)=0$ iff $\operatorname{Ext}_\Gamma^1(S_{\mu(\n)},S_{\mu(\m)})=0$. Since $\Lambda$ is quasicommutative, then we obtain that $\Gamma$ is also.
\end{proof}



    

We now consider fixed rings. We begin with a general lemma.

\begin{lemma}\label{qc-quotient}
    If $\Lambda$ is an algebra, $I$ and ideal of $\Lambda$, and $V$ a generalized weight module for $\Lambda/I$, it is generalized weight module for $\Lambda$ under the publlback $\Lambda \onto \Lambda/I$.
\end{lemma}
\begin{proof}
Clearly $\cfs \, \Lambda/I = \{ \m \in \cfs \, \Lambda| \m \supset I \}$, so $\cfs \, \Lambda/I \subset \cfs \, \Lambda$. If as a $\Lambda/I$-module, $V=\bigoplus_{\m \in \cfs \, \Lambda} V(\m)$, then $V$ has the same generalized weight space decomposition as a $\Lambda$-module.

\end{proof}

\begin{lemma}\label{all-the-time}
    Let $\Lambda$ be a Noetherian affine algebra over a field $\k$, and let $W$ be a finite group of algebra automorphisms of $\Lambda$, with $|W|^{-1} \in \k$. If $\mathfrak{M} \in \cfs \, \Lambda$, $\m = \mathfrak{M} \cap \Lambda^W$ is an maximal ideal of finite codimension as well in $\Lambda^W$.
\end{lemma}
\begin{proof}
    We have clearly an algebra map $\phi: \Lambda^W/\m \into \Lambda/\mathfrak{M}$. Since $\Lambda/\mathfrak{M}$ is finite dimensional, so is $\Lambda^W/\m$. It remains to check that $\m$ is a maximal ideal. If not, then we can choose a maximal ideal $\n$ with $\m \subsetneq \n$ and by Theorem \ref{LO}, a prime ideal $\mathfrak{N}$ in $\Lambda$ lying over it. But since $\mathfrak{M}$ is maximal, $\mathfrak{M}=\mathfrak{N}$, and so $\m=\n$ is maximal of finite codimension.
\end{proof}

The following is a generalization to quasicommutative algebras of the main result in \cite{KS} for FCRs.

\begin{theorem}\label{invariants-quasicommutative}
Let $\Lambda$ be a Noetherian affine algebra over a field $\k$, and let $W$ be a finite group of algebra automorphisms of $\Lambda$, with $|W|^{-1} \in \k$. If $\Lambda$ is quasicommutative, then $\Lambda^W$ is quasicommutative.
\end{theorem}

\begin{proof}

By Theorem \ref{MS}, $\Lambda$ is a finitely generated $\Lambda^W$-module, and the later ring is also affine and Noetherian. By Proposition \ref{Fillmore}, we have to show that each finite dimensional $\Lambda^W$-module $M$ is a generalized weight module. By Lemma \ref{qc-quotient}, we can assume $M$ to be a faifhfull $\Lambda^W$-module. Define $M_\Lambda = \Lambda \otimes_{\Lambda^W} M$. As $\Lambda$ is quasicommutative, $M_\Lambda$ is a generalized weight module. Let $\mathfrak{N} \in \operatorname{Supp} \, M$. $\n=\mathfrak{N} \cap \Lambda^W$ belongs to $\cfs \, \Lambda^W$ by Lemma \ref{all-the-time}. Let $v \in M(\mathfrak{N})$. $\{ 0 \}=\mathfrak{N}^{n(v)}v \supset \n^{n(v)} v$. So $M_\Lambda$ is a generalized weight module with respect to $\Lambda^W$. Hence, as $M$ is a  $\Lambda^W$-submodule of $M_\Lambda$, it is also a generalized weight space module.
\end{proof}


With this result, we have also the following remarkable

\begin{theorem}\label{remarkable}
    Let $U$ be an associative algebra and $\Lambda \supset \Gamma$ two quasicommutative algebras such that $\Lambda$ is a finitely generated left and right $\Gamma$-module, with $\Lambda$ and $\Gamma$ Noetherian. $\Lambda$ is a Harish-Chandra subalgebra if and only if $\Gamma$ is a Harish-Chandra subalgebra. Moreover, $\mathbb{H}(U, \Lambda)= \mathbb{H}(U, \Gamma)$. In particular, suppose $\Lambda$ is a Harish-Chandra subalgebra and let $W$ be any group of algebra automorphisms of $\Lambda$ whose order is invertible in $\k$. Then $\Lambda^W$ is a Harish-Chandra subalgebra and $\mathbb{H}(U, \Lambda)=\mathbb{H}(U, \Lambda^W)$. 
\end{theorem}
\begin{proof}
 If $\Lambda$ is a Harish-Chandra subalgebra, given any $u \in U$, $\Lambda u \Lambda$ is a finitely generated left and right module. $\Gamma u \Gamma$ is a sub-$\Gamma$-bimodule of $\Lambda u \Lambda$. Since the later is finitely generated left and right module for $\Lambda$, it is also going to be finitely generated left and right $\Gamma$-module as $\Gamma$ is Noetherian - and hence $\Gamma$ is Harish-Chandra. Suppose now $\Gamma$ Harish-Chandra. By hypothesis, $\Lambda=\sum_{i=1}^s \lambda_i \Gamma$, $\lambda_j \in \Lambda$, and as well $\Lambda=\sum_{i=1}^s  \Gamma \mu_i$, with $\mu_j \in \Lambda$. Let $u \in U$, and consider $\Lambda u \Lambda$. This expression equals $\sum_{i=1}^s \lambda_i (\sum_{j=1}^s \Gamma  u \Gamma \mu_j)$. As $\Gamma$ is Harish-Chandra, there are $v_1, \ldots v_k \in U$ with $\sum_{j=1}^s \Gamma  u \Gamma \mu_j=\sum_{\ell=1}^k \sum_{j=1}^s\Gamma v_\ell \mu_j$. Hence $\Lambda u \Lambda = \sum_{i=1}^s \lambda_i (\sum_{\ell=1}^k \sum_{j=1}^s \Gamma v_\ell \mu_j) \subset  \sum_{j,\ell} \Lambda v_\ell m_j$. Since $\Lambda$ is Noetherian, then $\Lambda u \Lambda$ is a finitely generated left $\Lambda$-module; and similarly on the right.

 Let now $M \in \mathbb{H}(U, \Lambda)$. Since $\Lambda$ is Noetherian and $M$ Harish-Chandra, $\Lambda$ act on it locally finitely; and hence so does $\Gamma$. If on the other hand $M$ is a finitely generated $U$-module where $\Gamma$ acts locally finitely, as $\Lambda$ is a finitely generated $\Gamma$-module, it will also act locally finitely. Hence $\mathbb{H}(U,\Lambda)=\mathbb{H}(U, \Gamma)$.
 
 By the previous theorem, we have that $\Lambda^W$ is quasicommutative. We also have $\Lambda^W$ Noetherian and $\Lambda$ a finitely generated $\Lambda^W$-module by Theorem \ref{MS}. So our last statement follows from what we have done previously. 

\end{proof}

It has been an vaguely stated open problem for more than a decade to find more ``natural" cases of pair algebra/Harish-Chandra subalgebra where $\Gamma$ is not necessarily commutative. Now we have plenty:

\begin{corollary}\label{mostrar}
In the setting of Theorem \ref{invariants-quasicommutative}, $\Lambda^W$ is a Harish-Chandra subalgebra of $\Lambda$. In particular, if $\mathfrak{g}$ is a finite dimensional reductive or nilpotent Lie algebra over an algebraically closed field of characteristic $0$ and $H$ any finite group of algebra automorphisms of $U(\mathfrak{g})$, $U(\mathfrak{g})^H$ is a Harish-Chandra subalgebra of $U(\mathfrak{g})$.
    
\end{corollary}
\begin{proof}
    $\Lambda$ is clearly a Harish-Chandra subalgebra of itself, so we can apply previous Theorem and the fact that $U(\mathfrak{g})$ is a quasicommutative algebra.
\end{proof}

We will see now another application of Theorem \ref{remarkable}, which coupled with Lemma \ref{Berest}, turns out to have a nice consequence for rational Cherednik algebras \cite{EG}.

In their situation, let $\k=\C$ and $W$ be a complex reflection group with representation $h$. We have, fixing $t=1$ and choosing a suitable deformation parameter $\mathfrak{c}$\footnote{which can be interpreted as just a function from the set of complex reflections to $\C$, invariant under $W$-conjugation \cite{EG}.}, we have the rational Cherednik algebras $H_\mathfrak{c}:=H_\mathfrak{c}(W,h)$, which are a universal flat family of deformations of $W_n(\C)*W$; they can also be considered as a double step degeneration of the corresponding DAHA \cite{EG}. We can  make with them an analysis that is similar to $W_n(\C)*G$, using the PBW Theorem for them and the Dunkl embedding \cite{EG}.

    It is well know (e.g, \cite{Vale}), that for most values of $\mathfrak{c}$\footnote{In fact, such values are Weil generic in a suitable affine space.}, the algebra $H_\mathfrak{c}$ will be affine, Noetherian and simple. Things are interesting when we consider values of $\mathfrak{c}$ for which the rational Cherednik algebra has irreducible finite dimensional modules. When $W=S_n$ and $h=\C^{n-1}$, this was done in \cite{BEG}. In this case $\mathfrak{c}$ is just a complex scalar, and finite dimensional irreducible modules exist when $\mathfrak{c}=\pm r/n$, with $r \in \mathbb{N}^+$ coprime with $n$. In this case there are only two irreducible finite dimensional modules with no non-splitting extension (cf. \cite[Theorem 1.2]{BEG}). This has a rather peculiar consequence:

    \begin{proposition}
    When $W=S_n$, $h=\mathbb{C}^{n-1}$ and $\mathfrak{c}$ is as above, $H_\mathfrak{c}(S_n,\mathbb{C}^{n-1})$ is a quasicommutative algebra.
    \end{proposition}

\begin{remark}
It is not true in general rational Cherednik algebras that admit finite dimensional irreducible modules are quasicommutative. For instance, we have non-split extensions for different finite dimensional modules even when $W=\mathbb{Z}_m$, $m \geq 3$. The author would like to thank Pavel Etingof for this simple yet important remark.
\end{remark}

In \cite{Berest} (cf. also \cite{Ginzburg2}) it was introduced th category $\mathcal{O}_{H_\mathfrak{c}}$ of modules for the rational Cherednik algebras: it consists of the finitely generated modules such that $\C[h^*]$ acts locally finitely. Since, by Chevalley-Shephard-Todd Theorem\footnote{It would me more correct, from an historical perspective, to atribute this theorem to Serre as well; for it was him who noticed that Chevalley proof worked not only for reflection groups, but also for pseudo-reflection groups \cite{Serre}.} (see  \cite[Chapter 5, $\oint$ 5, n. 5, Theorem 4]{Bourbaki}), $\C[h^*]$ is a free $\C[h^*]^W$-module of rank $|W|$, this is the same to require that $\C[h^*]^W$ acts locally finitely. By the Lemma \ref{Berest}, $\C[h^*]^W$ is a Harish-Chandra subalgebra of $H_\mathfrak{c}$. But by Theorem \ref{remarkable}, $\C[h^*]$ also is a Harish-Chandra subalgebra.

So category $\mathcal{O}_{H_\mathfrak{c}}$ is in fact a Harish-Chandra category:

\begin{theorem}\label{O-H-C}
    Let $H_\mathfrak{c}$ be a rational Cherendik algebra. Then $\mathcal{O}_{H_\mathfrak{c}}=\mathbb{H}(H_\mathfrak{c}, \C[h^*])$.
\end{theorem}

It is an open problem if $H_\mathfrak{c}$ can be realized as a Galois order with $\Gamma=\C[h^*]$. Such a realization was accomplished, using the methods of Coulomb branches \cite{BFN} for their spherical subalgebras in the case $W$ is of the type $G(m,p,n)$, in \cite{LW}. We remark, however, that in this case $\Gamma$ is a complicated object, not just $\C[h^*]^W$. However, even with $\mathbb{H}(H_\mathfrak{c}, \C[h^*])$ being the same as $\mathbb{H}(H_\mathfrak{c}, \C[h^*]^W)$, it is \emph{impossible} to have a Galois order structure with $\Gamma=\C[h^*]^W$, for this would imply that $\C[h^*]^W$ is maximal commutative with respect to inclusion (combine Theorem \ref{FO-H-bridge} and \cite[Proposition 2.14]{Hartwig}), which is clearly not the case.

\begin{remark}
    What is usually called category $\mathcal{O}$ for a rational Cherednik algebra (cf. \cite{GGOR}) is the block corresponding to $0$ in the decomposition of the above category (\cite{Berest}). We will discuss it in more detail bellow.
\end{remark}

After this interlude in the theory of rational Cherednik algebras, we go back to the general theory of quasicommutative and Harish-Chandra algebras and modules.

\begin{proposition}\label{non-Galois}
    Let $U$ be an associative algebra, $\Lambda$ an affine Noetherian Harish-Chandra subalgebra. Let $W$ be a finite group of automorphisms of $U$ whose order is invertible in $\k$, and such that such that $W(\Lambda) \subset \Lambda$. Assume as well that the group algebra $\k W$ is a subalgebra of $U$.
    
    \begin{enumerate}
    \item $\Lambda^W$ is a Harish-Chandra subalgebra of $eUe$, where $e=1/|W|\sum_{w \in W}$.
    
    \item If $M$ is a Harish-Chandra $U$-module, then $eM$ is a Harish-Chandra $eUe$-module.
    
    \item If $M$ is also uniform Harish-Chandra, then so is $eM$.
    
    \item If every $\Lambda$-generalized weight space for $M$ is finite dimensional, so is every $\Lambda^W$-generalized weight space.\footnote{Studies for which this is always the case for irreducible Harish-Chandra modules are an important theme \cite{DFO} \cite{FO2} \cite{Hartwig2} \cite{Fillmore}. }
    \end{enumerate}
    \end{proposition}
\begin{proof}
     By the Theorems \ref{invariants-quasicommutative} and \ref{remarkable}, $\Lambda^W$ is a Harish-Chandra subalgebra of $U$; and since $eUe$ is a subalgebra of $U$ which contains $\Lambda^W$\footnote{More precisely, it contains $e \Lambda^W e$, but since in $eUe$, $e$ acts as the unity, we are free to make this harmless identification.} it is Harish-Chandra in it also. This shows (1). Now, it is clear that $eM$ is a finitely generated $eUe$-module. Since $M$ is Harish-Chandra $U$-module and $\Lambda$ is Noetherian, this means that $\Lambda$ acts locally finitely on $M$. Since $\Lambda$ contains $\Lambda^W$ as a subalgebra, the later also acts locally finitely on $M$, and hence on $eM$. So we have dealt with (2). Suppose $M$ uniform with respect to $\Lambda$, and let $\n \in \operatorname{Supp} \, (eM)$. Let $\mathfrak{N}$ be a maximal ideal of finite codimension lying over it in $\Lambda$ (cf. Theorem \ref{LO}
     ) belonging to $\operatorname{Supp \, M}$. We have that for some $N>0$, $\mathfrak{N}^N M(\mathfrak{N}) = 0$. Hence $0=e \mathfrak{N}^N M(\mathfrak{N})=\n^N e M(\mathfrak{N}) \supset \n^N (eM)(\n)$ as well. This deals with (3).  Finally, let $\n \in \cfs \, \Lambda^W$. Let $\mathfrak{N}$ lie over it in $\Lambda$. $(eM)(\n) \subset e(\sum_{w \in W} M(w.\mathfrak{N}))$, where we are using Theorem \ref{LO}. The later is finite dimensional. Hence (4), and we are done.
\end{proof}

With essentially the same proof, we have also the following abstracted version of previous Proposition.

\begin{proposition}\label{abstract}
Let $U$ be an associative algebra, $\Lambda$ an affine Noetherian Harish-Chandra subalgebra. Let $e$ be an idempotent of $U$ such that $e\Lambda = \Lambda e$, and such this is a finitely generated $e \Lambda e$-module.

    \begin{enumerate}

    \item $e \Lambda e$ is Noetherian.
    
    \item $e\Lambda e$ is a Harish-Chandra subalgebra of $eUe$.
    
    \item If $M$ is a Harish-Chandra $U$-module, then $eM$ is a Harish-Chandra $eUe$-module.
    
    \item If $M$ is also uniform Harish-Chandra, then so is $eM$.
    
    \end{enumerate}
\end{proposition}
\begin{proof}
    Item (1) is \cite[21.11]{Lam}. The rest is the same as in previous Proposition.
\end{proof}

\begin{corollary}\label{invariants-strong-HC}
In the notation and conditions of Proposition \ref{non-Galois}, if $\Lambda$ is a strong Harish-Chandra subalgebra of $U$, then $\Lambda^W$ is a strong Harish-Chandra subalgebra of $eUe$.
\end{corollary}
\begin{proof}
    By the above Proposition, $\Lambda^W$ is a Harish-Chandra subalgebra of $eUe$.
    Let $\mathfrak{M} \in \cfs \, \Lambda$, and $\m = \mathfrak{M} \cap \Lambda^W \in \cfs \, \Lambda^W$. $eUe/ eUe \m$ is submodule of $e(U / U \mathfrak{M})$. The later module belongs, by the above Proposition, to $\mathfrak{H}(eUe, \Lambda^W)$. Hence, its submodule $eUe/ eUe \m$ is also a uniform Harish-Chandra module. Since all quotients $eUe/U \m$ appear in this way (cf. Theorem \ref{LO} and Lemma \ref{all-the-time}), and a similar argument works for right modules, we have that $\Lambda^W$ is a strong Harish-Chandra subalgebra.
\end{proof}

The following is the main result of this section (or this paper in general). Succinctly, if we have an algebra and an ``spherical subalgebra" that are Morita equivalent and have compatible Harish-Chandra subalgebras, then the Harish-Chandra module categories will be equivalent as well. It is also partially a generalization of some results in \cite{Webster} and \cite{Hartwig2}.

Before we state it, point out a lemma that, despite its simplicity, we found very worth pointing out.

\begin{lemma}
    Let $U$ be an associative algebra and $\Lambda$ a Harish-Chandra subalgebra of it. For every element $u$ of $U$, $\Lambda u$ is contained in a finitely generated right $\Lambda$-submodule of $U$.
\end{lemma}
\begin{proof}
    Clearly $\Lambda u \subset \Lambda u \Lambda$. The later is a finitely generated $\Lambda$-right submodule, as $\Lambda$ is Harish-Chandra
\end{proof}

\begin{theorem}\label{Morita-HC}
    In the setting of the Proposition \ref{non-Galois}, let $\mathbb{H}(U,\Lambda)$ be the category of Harish-Chandra modules for the pair $(U,\Lambda)$, and let $\mathbb{H}(eUe,\Lambda^W)$ be the category of Harish-Chandra modules for the pair $(eUe,\Lambda^W)$. We have a functor $F(\cdot):\mathbb{H}(U,\Lambda) \rightarrow \mathbb{H}(U^W, \Lambda^W)$, $F(M)=eM$ that is exact, where $e$ is the same idempotent as before; and we have a  functor $G(\cdot):=Ue \otimes_{eUe} (\cdot): \mathbb{H}(eUe, \Lambda^W) \rightarrow \mathbb{H}(U,\Lambda)$. If $UeU=U$, then $G$ is exact and then the pair  $(F, G)$ gives us an equivalence of categories $\mathbb{H}(U,\Lambda)$ and $\mathbb{H}(eUe,\Lambda^W)$.
\end{theorem}
\begin{proof}
    In the previous Proposition we obtained a functor $F:\mathbb{H}(U,\Lambda) \rightarrow \mathbb{H}(eUe, \Lambda^W)$. By Theorem \ref{Morita}, it is naturally isomorphic with tensoring with a projective module, hence it is exact. If $M$ is a Harish-Chandra module for $eUe$, we can define a functor $G(\cdot): \mathbb{H}(eUe,\Lambda^W) \rightarrow U-\operatorname{Mod}$ given by $G(M)= Ue \otimes_{eUe} M$. We need to show that $\Lambda m^*$ is finite dimensional for every $m^* \in G(M)$ in order to have $G(M)$ a Harish-Chandra module. $m^*$ is a linear combination of elements of the form $u \otimes m$, where $u \in Ue$ and $m \in M$. $\Lambda= \sum \lambda_i \Lambda^W$, for a finite number of fixed $\lambda_i \in \Lambda$. So $\Lambda.(u \otimes m)=\sum \lambda_i \Lambda^W(u \otimes m)=\sum \lambda_i \Lambda^We(u \otimes m)$\footnote{Remember that, under our harmless identification $e \Lambda^W e \equiv \Lambda^W$, $\Lambda^W e= \Lambda^W e e= e \Lambda^W e = \Lambda^W$.}$=\sum \lambda_i \Lambda^W(v \otimes m)$, where $v=eu \in eUe$. By the previous Lemma, $\sum \lambda_i \Lambda^W (v \otimes m)$ is a subspace of $ (\sum \lambda_i  (\sum v_\ell \Lambda^W)) \otimes  m$, for some finite number of $v_\ell \in eUe$. Call this vector space $V$. $V= (\sum \lambda_i (\sum v_\ell )) \otimes \Lambda^W m$. Since $M$ is a Harish-Chandra module for $eUe$ and $\Lambda^W$ is a Harish-Chandra subalgebra, $V$ is finite dimensional. So $\Lambda (u \otimes m)$ is also finite dimensional and $G(M)$ is a Harish-Chandra module for $U$ with respect to $\Lambda$. Now suppose $UeU=U$. Then $eUe$ and $U$ are Morita equivalent by Theorem \ref{Morita3}, and $Ue$ is a right projective $eUe$-module, which readly implies that $G$ is exact. The equivalence of the module categories given by the same functors $F$ and $G$ as here. We have, given the Morita equivalence, $F \circ G$ is naturally isomorphic to the identity functor of $U-\operatorname{Mod}$; and $G \circ F$ naturally isomorphic to the identity functor in $eUe$. These natural isomorphisms, of course, persists when we restrict to the Harish-Chandra module categories. Hence, this restriction gives us the desired equivalence.
\end{proof}

We recall that, given a rational Cherednik algebra $H_\mathfrak{c}$, its spherical subalgebra is $U_\mathfrak{c}:=e H_\mathfrak{c} e$. We can introduce a category $\mathcal{O}_{U_\mathfrak{c}}$ for the spherical subalgebra by considering the finitely generated modules where $\mathbb{C}[h^*]^W$ act locally finitely. It is known that $H_\mathfrak{c} e H_\mathfrak{c}=H_\mathfrak{c}$ if and only if $\mathfrak{c}$ is generic (and the Cherednik algebra is then simple), if and only if $H_\mathfrak{c}$ and $U_\mathfrak{c}$ are Morita equivalent (\cite{EG}; see also \cite{Vale} \cite{Bellamy0})  By the previous result, we easily get

\begin{corollary}
    When $\mathfrak{c}$ is generic, the categories $\mathcal{O}_{H_\mathfrak{c}}$ and $\mathcal{O}_{U_\mathfrak{c}}$ are equivalent module categories.
\end{corollary}

We finish that we can obtain an analogue of Theorem \ref{Morita-HC} if we work in the setting of Proposition \ref{abstract}: just replace, in the statement and proof of Theorem \ref{Morita-HC} $\Lambda^W$ by $e\Lambda e$ everywhere. To state it formally:

\begin{theorem}\label{abstract2}
    In the setting of Proposition \ref{abstract}, we have the same result as that of Theorem \ref{Morita-HC}, with the same proof, by replacing every ocurrence of $\Lambda^W$ by $e \Lambda e$.
\end{theorem}



\section{Galois rings and orders II}

We now specialize our discussion for Harish-Chandra categories $\mathbb{H}(U, \Gamma)$ where $U$ is a Galois $\Gamma$-ring in the fixed ring $\K$ of an adequate skew monoid ring $\mathcal{L}$.

\subsection{General results}

By the main results of \cite{FO2}, we are free to use the results just obtained in the previous section.

Our first objective is to show that, as long $U$ is a Galois $\Gamma$-order and $\Gamma$ is a integrally closed and Harish-Chandra subalgebra, which is the minimal condition for the the two settings for Galois rings in \cite{FO} and \cite{Hartwig} to coincide, they do coincide.

\begin{lemma}
    Let $\Gamma$ be a domain, $K$ its field of fractions, $L$ a finite extension of $K$ and $\Lambda$ the integral closure of $\Gamma$ in $L$. Then $\Frac \, \Lambda = L$.
\end{lemma}
\begin{proof}
    Let $x \in L$. As $L$ is algebraic over $K$, after clearing denominators, we have that $x$ satisfies a polynomial with coefficients in $\Gamma$. If $a$ is the leading coefficient of this polynomial, $ax$ is clearly integral over $\Gamma$, and hence for some $y \in \Lambda$, $ax=y$. This shows the statement.
\end{proof}

\begin{theorem}\label{FO-H-bridge}
    \emph{(FO-H bridge)} Let $U$ be a Galois $\Gamma$-ring in $\mathcal{K}=\mathcal{L}^G$, $\mathcal{L}=L*\mathcal{M}$. If $\Gamma$ is a Harish-Chandra subalgebra of $U$ and an integrally closed domain, then $U$ satisfies the Hartwig conditions with $\Lambda$ the integral closure of $\Gamma$ in $L$.  
\end{theorem}
\begin{proof}
    By \cite[Proposition 5.1]{FO}, $\mathcal{M}.\Lambda=\Lambda$. Let us show now that $G(\Lambda) \subset \Lambda$. If $x \in \Lambda$ satisfies $x^n + \sum_{i=0}^{n-1} \gamma_i x^i=0$, applying any $g \in G$ gives us, since $\Gamma \subset K=L^G$, that $(g.x)^m + \sum_{i=0} \gamma_i (g.x)^i=0$. So $g.x \in \Lambda$. $\Lambda^G$ is integrally closed \cite[Lemma 2.1]{Hartwig} and $\Frac \, \Lambda^G=L^G=K$ by the previous lemma. Hence $\Gamma \subset \Lambda^G \subset K$, and each element of $\Lambda^G$ is killed by a monic polynomial with coefficients in $\Gamma$, by definition. As $\Gamma$ is its own integral closure in $K$, $\Lambda^G=\Gamma$.
\end{proof}

In this way, we can move freely between the setting of \cite{FO} and that of \cite{Hartwig}. It is also worth remarking that the condition for a pair algebra/subalgebra be a Galois ring in the sense of Futorny-Ovsienko \cite{FO} and in the sense of Hartwig \cite{Hartwig} does not depend on $U$, $\Lambda$, neither on the embedding on th skew group ring, but just on $\Gamma$.

\medskip

In \cite{Webster}, a notion of flag order was introduced. We will not define it, since we will not use it. However, we point out that we have the following improvement of \cite[Lemma 2.3]{Webster}.

\begin{theorem}\label{non-principal-flags}
    Let $U$ be a Galois order over $\Lambda$ in $L * (\mathcal{M}\rtimes W)$, where $L= \Frac \, \Lambda$, $W$ is a finite group of $\operatorname{Aut}_k \, \Lambda$ and $\mathcal{M}$ a subgroup of $\operatorname{Aut}_k \, \Lambda$, such that $W$ that acts faithfully on $\mathcal{M}$, normalizing it and $\mathcal{M} \cap W = 1$. Let $e=1/|W| \sum_{w \in W} w$. Then $eUe$ is a Galois order over $\Lambda^W$ in $(L*\mathcal{M})^W$.
\end{theorem}
\begin{proof}
    By \cite[Lemma 2.1]{Montgomery}, $e \Lambda e \simeq \Lambda^W$ and $e (L*(\mathcal{M} \rtimes W))e \simeq (L*\mathcal{M})^W$. By hypothesis, $LU=UL=L*(\mathcal{M} \rtimes W)$, and hence $L^W e Ue = e L^W U e = e L U e= e U L e= e U e L^W = (L*\mathcal{M})^W$.  If we show that $\mathcal{M}$ is separating for $L^W$, we will have that $eUe$ is Galois $\Lambda^W$-ring. But this follows from our hypothesis, since $\mathcal{M} \cap W = 1$, using Lemma \ref{separating}

    We will show now that $eUe$ is a Galois order using the definition. So let $V$ be a finite dimensional left vector space for $L^W$ in $(L*\mathcal{M})^W$. Let $V'=LV$ be the finite dimensional left vector space generated by $V$ over $L$ in $L*(\mathcal{M} \rtimes W)$. Since $U$ is a Galois order, $V' \cap U$ is a finitely generated left $\Lambda$-module, and hence a finitely generated left $\Lambda^W$-module. $V \cap eUe$ is a $\Lambda^W$-submodule of $V' \cap U$, and since $\Lambda^W$ is Noetherian, it is also finitely generated as a left module. Arguing similarly on the right, we have that $eUe$ is a Galois order. If $\Lambda$ is Harish-Chandra in $U$ then $\Lambda^W$ is Harish-Chandra in $eUe$ by Proposition \ref{non-Galois}.
\end{proof}

Principal Galois orders are important because of:

\begin{proposition}
    If $U$ is a principal $\Gamma$-order, and $\Gamma$ is integrally closed, for each $m \in \operatorname{Specm} \Gamma$, we have a \emph{canonical} simple Harish-Chandra module with $m$ in its support.
\end{proposition}
\begin{proof}
    This follows from \cite[Theorem 3.3, Definition 3.4]{Hartwig}, and Theorem \ref{FO-H-bridge}.
\end{proof}



We finish this section with an important Proposition that will be very handy later on.

\begin{proposition} \label{main-prop}
Let $U$ be a Galois algebra over $\Gamma$, and $G$ be a finite group of automorphism of $U$ such that $G(\Gamma) \subset \Gamma$. If $U^G$ is Galois ring over $\Gamma^G$, then $U^G$ is a principal Galois order.

\end{proposition}

\begin{proof}
    Call $K = \Frac \, \Gamma$. Since $U^G \subset U$ and the later is a principal Galois order, then $U^G(\Gamma) \subset \Gamma$. On the other hand, by \cite[Lemma 2.19]{Hartwig}, $U^G(K^G) \subset K^G$. Hence the claim follows.
\end{proof}

\subsection{Multiplicative invariant theory of the algebra of differential operators on the torus}

We begin this subsection explaining, for the sake of completeness, what multiplicative invariant theory is. In classical invariant theory, you have a finite dimensional vector space $V$, a subgroup $G$ of $\operatorname{GL}(V)$ and ask questions about the invariants of $S(V)$. For a great introduction to subject, we refer to \cite{DK}.

In multiplicative invariant theory, we have a finite rank lattice $M$ (that is, a free abelian group of finite rank), and a subgroup $G  < \operatorname{GL}(M)$. In this context, $M$ is called a $G$-lattice. The group ring $\k[M]$, which is clearly isomorphic to a ring of Laurient polynomials, carries a natural faithfull $\k$-algebra action of $G$, also sometimes called monomial action, and the question center around the structure of $\k[M]^G$. The standard reference for this topic is \cite{Lorenz}.

\begin{definition}\label{invariant-differential-operators}
    Let $X$ be an affine irreducible affine variety with the action of a finite group $G$. Then the action of $G$ extends to an action on the ring of differential operators by conjugation: if $D \in \mathcal{D}(X), f \in \mathcal{O}(X), g \in G$, $g.D(f):= g(D(g^{-1}f))$. 
\end{definition}

In \cite[Section 4.4]{FS3} we considered certain invariants of differential operators on the torus. We can reinterpret what was done there using the notion of multiplicative invariants.

Consider, for instance, the action of $B_n$ on $\mathcal{D}(\k^{\times n})$ considered in \cite{FS3}. $\Z^n$ has the structure of a $B_n$-lattice, which is given by just signed permutations. Then we recover the action on the ring of differential operator on the torus, using Definition \ref{invariant-differential-operators}, by considering the $B_n$ action on $\mathcal{D}(\k[\Z^n])=\mathcal{D}(\k^{\times n})$.

Let's expand now the technical details.

Consider the ring of differential operators on the 1-torus $\mathcal{D}(\k[x^{\pm 1}])$. This is the algebra generated over $\k$ by $x,x^{-1}$ and $\partial=d/dx$. Call $t=\partial x$. We also have that the ring of differential operators on $\k^\times$ is the algebra generated $x,x^{-1}, t$. We have an isomorphism $f: \mathcal{D}(k[x^{\pm 1}]) \rightarrow k[t]* \mathbb{Z}$, given by $x \mapsto \varepsilon$, $x^{-1} \mapsto -\varepsilon$, $t \mapsto t$, where the action is $\varepsilon(t)=t-1$ \cite{FO}.

Generalizing:

\begin{proposition}[\cite{FO}]\label{diff-ops}
    Let $\mathcal{D}(\k^{\times n})=\mathcal{D}(k[x_1^{\pm 1}, \ldots, x_n^{\pm n}])$ be the algebra of differential operators on the n-torus. Call $t_i=\partial_i x_i$.
    $\mathcal{D}(\k^{\times n})$ is generated, as an algebra, by $x_1^{\pm 1}, \ldots, x_n^{\pm n}, t_1, \ldots, t_n$.

    There is an isomorphism 

    \[ f: \mathcal{D}(\k^{\times n}) \rightarrow k[t_1,\ldots, t_n]*\mathbb{Z}^n \]

    which sends $x_i$ to $\varepsilon_i$, $x_i^{-1}$ to $-\varepsilon_i$, $t_i$ to $t_i$, with the action of $\varepsilon_i$ being $\varepsilon_i(t_j)=t_j - \delta_{ij}$. Hence $\mathcal{D}(k^{\times n})$ is a principal Galois order over $\Gamma=k[t_1,\ldots,t_n]$ in $k(t_1,\ldots,t_n)*\mathbb{Z}^n$.
\end{proposition}
\begin{proof}
    The only thing left to show is that is the ring of differential operators is a principal Galois order. But this is clear, as in the embedding $f$ there are no denominators.
\end{proof}

Our generalization is as follows. Suppose $L$ is a $G$-lattice, for a finite group $G$. We have $\mathcal{D}(\k[L])^G \simeq (\k[t_1,\ldots,t_n]*\mathbb{Z}^n)^G$. It is clear that $G(\mathbb{Z}^n)=\mathbb{Z}^n$; and since the action of $G$ on a differential operator is, by definition, an action by conjugation, and since the isomorphism $\mathcal{D}(\k[L]) \simeq \k[t_1,\ldots,t_n]*\mathbb{Z}^n$ is $G$-equivariant, we also have that $G$ acts on $\mathbb{Z}^n$ by conjugation.

\begin{theorem}\label{multiplicative-Galois}
    The invariant of the ring of differential operators on the torus $\mathcal{D}(\k[L])^G$ is a principal Galois order in $(\k(t_1,\ldots,t_n)*\mathbb{Z}^n)^G$ if $G(\k[t_1,\ldots,t_n]) \subset \k[t_1,\ldots,t_n]$, in which case the Harish-Chandra subalgebra is $k[t_1,\ldots,t_n]^G$.
\end{theorem}
\begin{proof}
    It is clear that we have $\mathcal{D}(\k[L])^G \simeq (\k[t_1,\dots,t_n]*\Z^n)^G \into (\k(t_1,\ldots,t_n)*\Z^n)^G$. Before proceeding, we recall that the ring of differential operators on any smooth variety is simple Noetherian finitely generated algebra and a domain (\cite[Chapter 15]{McConnell}). By Theorem \ref{MS}, $\mathcal{D}(\k[L])^G$ is a finitely generated algebra, by a finite set of generators $S$. Consider the elements $X_i^{\pm}=\sum_{g \in G} g(x_i^{\pm 1}), i=1,\ldots,n$. Adjoin them to the set $S$ to obtain a set $S'$, The support of the elements of $S'$ generate $\Z^n$ as a monoid, and so by Proposition \ref{prop-supports}, $\mathcal{D}(\k[L])^G$ is a $\Gamma^G$-ring in  $(\k[t_1,\dots,t_n]*\Z^n)^G$, where $\Gamma=\k[t_1,\dots.t_n]$. By the previous Proposition \ref{diff-ops} and Proposition \ref{main-prop}, $\mathcal{D}(\k[L])^G$ is a principal Galois order with Harish-Chandra subalgebra $\Gamma^G$.
\end{proof}

A natural question is, then, when do we have
$G(\k[t_1, \ldots, t_n]) \subset \k[t_1, \ldots, t_n]$? To decide on this question, we note the following elementary but useful proposition:

\begin{proposition}
    $\mathcal{D}(\k^{\times n})$ has a basis given by monomials of the form $x^\alpha y^\beta$, $\alpha \in \Z^n$, $\beta \in \mathbb{N}^n$. $\k[t_1, \ldots, t_n]$ coincides with the vector space generated by the monomials such that $\alpha_i + \beta_i$ is a nonnegative even integer, for $i=1, \ldots, n$.
\end{proposition}

As an ilustration (cf. \cite{FS3}), $x^2 y x^{-1} \in \mathcal{D}(\k^\times)$ is an element of $\k[t]$ by the above criterion: indeed, $x^2 y x^{-1}=[x^2, y] x^{-1}+yx^2 x^{-1}=-2+t$. The general problem of understanding the finite subgroups of $\operatorname{GL}_n(\mathbb{Z})$, however, is very difficult (cf. \cite{Lorenz}).

We discuss now the Gelfand-Kirillov Hypothesis and its q-analogue (cf. \cite[I.2.11, II.10.4]{Brown}). In \cite{FS} it was shown that, in case we have a linear action of a finite group $G$ on a finite dimensional vector space, that if $V/G$ is a rational variety, then $\Frac \, D(V)^G \simeq \Frac \, W_n(\k)$, where $n=\operatorname{dim}(V)$. In \cite{SchwarzPan} this was generalized to arbitrary group actions, so in particular, if $\k(L)^G$ is rational, then, calling $n=\operatorname{rank}(L)$,  $\Frac \, \mathcal{D}(\k[L])^G \simeq \Frac \, W_n(\k)$. We can obtain also a different result in the same vein, generalizing \cite[Theorem 3]{EFOS}. First we will generalize \cite[Theorem 6.1]{FS}.

\begin{theorem}\label{GK-1}
    Let $L=\k(x_1,\ldots,x_n)$, $\Z^n$ a group of automorphisms of $L$ such that $\varepsilon_i(x_j)=x_j-\delta_{ij}$. Let $G$ be \emph{any} finite group of automorphisms of $L$, which acts by conjugation on $\mathbb{Z}^n$, normalizing it. If $L^G$ is a purely transcendental extension, and $U$ is a Galois ring in $(L*\Z^n)^G$, then $\Frac \, U \simeq \Frac \, W_n(\k)$.
\end{theorem}
\begin{proof}
    By  \cite[Theorem 4.1(4)]{FO}, we have that $\Frac \, U \simeq \Frac \, (L*\Z^n)^G$. As is well known (cf. \cite{FMO}), $L*\Z^n$ is just a localization of the Weyl algebra. Hence $\Frac \, (L*Z^n)^G = \Frac \, W_n(\k)^G$. The result now follows from \cite[Theorem 0.12]{SchwarzPan}.
\end{proof}

\begin{corollary}
    In the setting of Theorem \ref{multiplicative-Galois}, if $\k(t_1,\ldots,t_n)^G$ is a purely transcendental extension, then $\Frac \mathcal{D}(\k^{\times n})^G \simeq \Frac W_n(\k)$.
\end{corollary}

We notice that we have two very different phenomena here. $\mathcal{D}(\k^{\times n})$ has two very distinct maximal commutative subalgeras: $\mathcal{O}(\k^{\times n})$, which is the algebra of Laurient polynomials, which is also ad-nilpotent; and $\Gamma = \k[t_1,\ldots,t_n]$, which is not ad-nilpotent.

By the general theme ``Noether's problem implies its noncommutative analogue" (cf. \cite{FS}, \cite{SchwarzPan}) the rationality of $\k(L)^G$ implies that $\Frac \, \mathcal{D}(\k[L])^G \simeq \Frac \, W_n(\k)$. The fact that Noether's problem for $\Gamma^G$ implies the Gelfand-Kirillov Conjecture for $\mathcal{D}(\k[L])^G$ is also related to such ideas, but use in an essential way the realization of $\mathcal{D}(\k^{\times n})$ as a Galois $\Gamma$-ring and the fact that if we have an associative $U$ Galois $\Gamma$-ring in $\K$, then $\Frac \, U \simeq \Frac \, \K$ (cf. \cite[Theorem 4.1(1)]{FO}). It seems an interesting question to be explored further how the rationality of $\k[x_1, \ldots, x_n]^G$ (or $\k[x_1^\pm, \ldots, x_n^\pm]^G)$ and that of $\Gamma^G$ are related.

\subsection{Fixed rings of GWA}

In this subsection, we will show that given any tensorial GWA (cf. Definition \ref{tensorial}), its invariants under a permutation action and its invariants under the action of the complex reflection groups of type $G(m,p,n)$ are principal Galois orders. In particular, our results will generalize all those of \cite{FS2} and \cite{FS3}. We will also discuss freeness over the Harish-Chandra subalgebra and their skew field of fractions. In this subsection, all algebras are commutative over an algebraically closed field of zero characteristic.

\subsubsection{Permutation groups}

Let $D$ be a finitely generated algebra which is a domain, $D(a,\sigma)$ a rank 1 GWA with infinite order automorphism, and let $D(a,\sigma)^n$ be the algebra $D(a,\sigma)^{\otimes n}$. The base ring of $D(a,\sigma)^n$ is $D^{\otimes n}$, and the other parameters are as described in Proposition \ref{prop-basic-GWA}. Call $F = \Frac \, D^{\otimes n}$.

We have a natural action of $S_n$ on $D(a,\sigma)^n$: $\pi X_i^{\pm}=X_{\pi(i)}^\pm; \pi (d_1 \otimes \ldots \otimes d_n) =(d_{\pi(1)} \otimes \ldots \otimes d_{\pi(n)} $, $\pi \in S_n$. The embedding of Theorem \ref{theorem-GWA}, $D(a,\sigma)^n \into F*\mathbb{Z}^n$ ,is $S_n$-equivariant, and the $S_n$-action on $\mathbb{Z}^n$ is by conjugation. Hence

\begin{theorem}\label{permutation}
    Let $G \leq S_n$, and consider the natural action by algebra automorphisms on $D(a,\sigma)^n$. Then $(D(a,\sigma)^n)^G$ is a principal Galois order in $(F*\Z^n)^G$ 
\end{theorem}
\begin{proof}
    Again by Theorem \ref{MS}, the invariant subalgebra $(D(a,\sigma)^n)^G$ is finitely generated as algebra by a finite set $S$, which we can harmless assume that contains the elements $\sum_{i=1}^n X_i^+$ and $\sum_{i=1}^n X_i^-$. Hence the support of the elements of $S$ generates $\Z^n$ as monoid, and by Proposition \ref{prop-supports}, $(D(a,\sigma)^n)^G$ is a $\Gamma$-Galois ring, where $\Gamma=(D^{\otimes n})^G$, in $(F*\Z^n)^G$. By Theorem \ref{theorem-GWA} and Proposition \ref{main-prop}, we have that the fixed subalgebra is a principal Galois order.
\end{proof}

In general, the theory of D-modules in focused on differential operators over smooth varieties (cf. \cite{Hotta}). But our study reveal that in certain cases we have a nice module category (i.e., Harish-Chandra) for rings of differential operators on singular varieties.

We remark that, it follows from Chevalley-Shephard-Todd Theorem that if $G<\operatorname{GL}_n(\k)$ is not a pseudo-reflection group, then the variety $\k^n/G$ is singular. We also have a famous result due to Levausseur \cite[Theorem 5]{Levasseur}: if $G$ does not contain any pseudo-reflections, we have $W_n(\k)^G \simeq \mathcal{D}(\k^n/G)$.

Hence

\begin{corollary}
    If $G \leq S_n$ is a finite permutation group that does not contain any pseudo-reflection, then $\mathcal{D}(\k^n/G)$ is a principal Galois order with Harish-Chandra subalgebra $\k[t_1,\ldots,t_n]^G$ in $(\k(t_1,\ldots,t_n)*\mathbb{Z}^n)^G$
\end{corollary}
\begin{proof}
    This follows from the above discussion and Theorem \ref{permutation}, as $W_n(\k)=W_1(\k)^{\otimes n}$.
\end{proof}

As an example of the above Corollary, noted first in \cite{FS3}, calling $\mathcal{A}_n$ the alternating groups with their usual permutation actions, $W_n(\k)^{\mathcal{A}_n} \simeq \mathcal{D}(\k^n/\mathcal{A}_n)$. The category of Harish-Chandra modules provides a natural category of $D$-modules for this singular variety.

\subsubsection{Complex reflection groups}

Let's remind the construction of the irreducible complex reflection groups $G(m,p,n)$. Let $G_m \subset \k$ be the cyclic group in $m$ elements of m-th roots of unity. Let $A(m,p,n)$, with $p|m$, be the subgroup of $G_m^{\times n}$ consisting of elements $(h_1, \ldots, h_n)$ such that $(\prod_{i=1}^n h_i)^{m/p}=id$. If we take the semi-direct product $A(m,p,n) \rtimes S_n$, where $S_n$ permutes the entries of $A(m,p,n)$, we obtain the irreducible complex reflection groups $G(m,p,n)$. In this section we will be interested, first, in the cyclotomic case, where $p=1$. In case of Weyl algebras, this was the main content of \cite{FS3}.

If $D(a,\sigma)$ is a rank 1 GWA with infinite order automorphism, it is clear that the action of $G_m$ described in Theorem \ref{JW} and its result can be generalized immediately:

\begin{proposition}
    Let $G_m$ acts on the tensorial GWA $D(a,\sigma)^n$ diagonally. Then $(D(a,\sigma)^n)^G$ is again a tensorial GWA: $(D(a*,\sigma^m))^n$, where we are using the same notation as in Theorem \ref{JW}. Hence $(D(a,\sigma)^n)^G$ is a principal Galois order with $\Gamma=D^{\otimes n}$ in $F*\Z^n$, where $F= \Frac D^{\otimes n}$.
\end{proposition}
\begin{proof}
    Follows from the combination of Proposition \ref{prop-basic-GWA}, Theorem \ref{theorem-GWA} and Theorem \ref{JW}.
\end{proof}

Through this section we will keep the notation $D(a*,\sigma^m)$ for the invariants $D(a,\sigma)^{G_m}$ of a rank 1 GWA.

\begin{theorem}\label{cyclotomic}
    Consider a tensorial GWA $D(a,\sigma)^n$ and its fixed ring $(D(a,\sigma)^n)^{G(m,1,n)}$. Then, in the notation of the previous Proposition, $(D(a,\sigma)^n)^{G(m,1,n)}$ is principal Galois order in $(F*\Z^n)^{S_n}$ with Harish-Chandra subalgebra $(D^{\otimes n})^{S_n}$.
\end{theorem}
\begin{proof}
    The first step is the easy remark that if $R$ is a ring, $G$ a group of ring automorphisms of it and $N$ a normal subgroup of $G$, then $R^G \simeq (R^N)^{G/N}$. So $(D(a,\sigma)^n)^{G(m,1,n)} \simeq (D(a*,\sigma^m)^n)^{S_n}$. We can now apply Theorem \ref{permutation} to conclude.
\end{proof}

We now generalize this result to all groups $G(m,p,n)$, which we remark is a subgroup of $G(m,1,n)$. We need first a lemma, whose proof is essentially the same as of \cite[Proposition 30]{FS3}.

\begin{lemma}
$(D(a,\sigma)^n)^{G(m,p,n)}=\bigoplus_{i=0}^{p-1} (X_1^+ \ldots X_n^+)^{mp/k} (D(a,\sigma)^n)^{G(m,1,n)}$    
\end{lemma}

\begin{theorem}\label{main-objective}

$(D(a,\sigma)^n)^{G(m,p,n)}$ is a principal Galois order with Harish-Chandra subalgebra $(D^{\otimes n})^{S_n}$ inside $(F *\mathbb{Z}^n)^{S_n}$. If $D=\k[h]$, then the algebra is a rational and co-rational Galois order.
\end{theorem}
\begin{proof}
    Since the action of $G(m,1,n)$ normalizes $\mathbb{Z}^n$ and is separating, the same can be said about the action of $G(m,p,n)$. Moreover the elements $\sum_{i=1}^n (X_i^+)^n$ and $\sum_{i=1}^n (X_i^-)^n$ belongs to $(D(a,\sigma)^n)^{G(m,p,n)}$ and their supports generate $\mathbb{Z}^n$ as a monoid. Hence we have that $(D(a,\sigma)^n)^{G(m,p,n)}$ is a Galois $(D^{\otimes n})^{S_n}$-ring in $(F *\mathbb{Z}^n)^{S_n}$. Since as we just saw $D((a,\sigma)^n)^{G(m,1,n)}$ is a principal Galois order, calling $\Gamma=(D^{\otimes n})^{S_n}$, by the previous Lemma $D((a,\sigma)^n)^{G(m,p,n)}(\Gamma) = \bigoplus_{i=0}^{p-1} (X_1^+ \ldots X_n^+)^{mp/k} (D(a,\sigma)^n)^{G(m,1,n)}(\Gamma)  
\subset \Gamma$ also, and so it is also a principal Galois order. The last claim is clear as there are no denominators in our embedding into $F*\Z^n$.
\end{proof}

As a corollary we obtain a generalization of the main result from \cite{FS3}.

\begin{corollary}\label{Gabi}
    Consider $G(m,p,n)$ with its natural action on $\k^n$, and extend it to algebra automorphisms of $W_n(\k)$ and $\mathcal{D}(\k^{\times n})$\footnote{the latter is just the localization of the Weyl algebra at the multiplicatively closed set generated by $x_1 x_2 \ldots x_n$}. This makes $W_n(\k)^{G(m,p,n)}$ and $\mathcal{D}(\k^{\times n})^{G(m,p,n)}$ principal Galois orders over the Harish-Chandra subalgebra $\k[t_1,\ldots,t_n]^{S_n}$ in $(\k(t_1,\ldots,t_n)*\Z^n)^{S_n}$. They are, in fact, rational and co-rational Galois orders.
\end{corollary}


The same result was obtained in \cite{LW} in the case of the Weyl algebras in the broader context of rational Cherednik algebras, but our method relying on invariant theory of generalized Weyl algebras is new and elementary.

We remark that the combination of the last Corollary, the main result of \cite{FS} and \cite[Theorem 4.1(4)]{FO} gives a new proof of the noncommutative Noether's problem for the complex reflections group of type $G(m,p,n)$, which is a particular case of the main result of \cite{EFOS}.

Finally, similar to \cite[Problem 5.1]{H0}, we have:

\begin{problem}
    Is it possible realize $W_n(\k)^G$ as a Galois ring or order, when $G$ is one of the 34 exceptional complex reflection groups?
\end{problem}

\subsection{Freeness and Gelfand-Kirillov Hypothesis}

It is an important question the the theory of Galois orders to decide when the algebra $U$ is a free left and right module over its Harish-Chandra subalgebra. A relevant theorem in this direction is a famous result by H. Bass \cite{Bass}, which we quote here in a form that is sufficient for our needs.

\begin{proposition}\label{Bass}
    If $\Gamma$ is a finitely generated algebra over $\k$ which is also a domain, and $M$ is a non-finitely generated projective $\Gamma$-module, then $M$ is actually free.
\end{proposition}

We are going to recall an useful lemma (\cite[Lemma 11]{FS3})

\begin{lemma}\label{main-lemma-1}

Let $U$ be an associative algebra and $\Gamma \subset U$ an affine commutative subalgebra. Let $G$ be a finite group of automorphisms of $U$ such that $G(\Gamma) \subset \Gamma$. If $U$ is a projective $\Gamma$-module and $\Gamma$ is a projective $\Gamma^G$-module, then $U^G$ is a projective $\Gamma^G$-module.
    
\end{lemma}

The two main algebras that appear as the base ring of a tensorial GWA are the polynomial algebra $\Lambda:=\k[h_1, \ldots, h_n]$, and the algebra of Laurient polynomials, which is just the localization $\Lambda[e_n^{-1}]:=\k[h_1, \ldots, h_n][e_n^{-1}]$, where $e_n=h_1 h_2 \ldots h_n$ is an elementary symmetric polynomial. It is a classical result from invariant theory that the algebra of symmetric polynomials $\Lambda^{S_n}$ is again polynomial and $\Lambda$ is a free $\Lambda^{S_n}$-module of rank $n!$. This can be generalized to reflection groups, by the Chevalley-Shephard-Todd Theorem. The same holds for Laurient polynomials:

\begin{lemma}\label{Laurient}
$\Lambda[e_n^{-1}]$ is a free $\Lambda[e_n^{-1}]^{S_n}$-module of rank $n!$. More generally, if $G \leq S_n$ is a reflection group, $\Lambda[e_n^{-1}]$ is a free $\Lambda[e_n^{-1}]^{G}$-module of rank $|G|$.
    
\end{lemma}
\begin{proof}
Since $e_n$ is $S_n$ invariant, $\Lambda[e_n^{-1}]^{S_n} \simeq \Lambda^{S_n}[e_n^{-1}]$. Since localization is an exact functor, we are done. The general case follows in the same way, by Chevalley-Shephard-Todd Theorem.
\end{proof}

\begin{theorem}
    Let $D(a,\sigma)$ be a tensorial GWA of rank n with $D$ a polynomial algebra in n indeterminates or an algebra of Laurient polynomials. If $G\leq S_n$ is a permutation group and we give $D(a,\sigma)^G$ the structure of a principal Galois order as in Theorem \ref{permutation}; or if we consider $D(a,\sigma)^G$, where $G=G(m,p,n)$ as in Theorem \ref{main-objective}, then $D(a,\sigma)^G$ is a free module over its Harish-Chandra subalgebra.
\end{theorem}
\begin{proof}
    $D(a,\sigma)$ is a free module over its Harish-Chandra subalgebra $D$. Calling $\Gamma$ the Harish-Chandra subalgebra of $D(a,\sigma)^G$, either by Chevalley-Shephard-Todd Theorem or Lemma \ref{Laurient}, we have that $D$ is a projective $\Gamma$-module. The result now follows from Proposition \ref{Bass} and Lemma \ref{main-lemma-1}
\end{proof}

Now we discuss the Gelfand-Kirillov-Hypothesis

\begin{theorem}\label{GK-2}
    If $D(a,\sigma)$ is a GWA of rank I and classical type (cf. Corollary \ref{GWA-simplicity-2}), then $\Frac \, (D(a,\sigma)^n)^{G(m,p,n)} \simeq \Frac W_n(\k)$. In case it is a GWA of rank I of quanum type, $\Frac \, (D(a,\sigma)^n)^{G(m,p,n)} \simeq \Frac(\k_{q^{m/p}}[x,y] \otimes k_q[x,y]^{\otimes n-1})$
\end{theorem}
\begin{proof}
    In the first case, we have the skew field of fractions of $(\k(t_1,\ldots,t_n)*\Z^n)^{S_n}$, which is known to be $\Frac \, W_n(\k)$ (cf. \cite{FMO}, \cite{FS}). In the second case, it follows from \cite[Theorem 1.1]{H0}
\end{proof}

\section{Some concrete Harish-Chandra modules}

This section has three subsections. In the first one we show that the category $\mathcal{O}$ for rational Cherednik algebras \cite{GGOR}, in the case of the invariants of the Weyl algebra discussed above (i.e., when the parameters $\mathfrak{c}=0$ and $t=1$ and $G=G(m,p,n)$), is contained in the category of Harish-Chandra modules, when $\Gamma=\k[t_1,\ldots,t_n]$. This is well known for the case of Gelfand-Tsetlin modules for $\mathfrak{gl_n}$ \cite{DFO}, and was explored in the case of of groups of type $G(m,p,n)$ in \cite{LW}, relying heavily on the work \cite{BFN}. Our exposition, in the case of Weyl algebra, however, is completely elementary, and includes nice examples of Harish-Chandra modules .

The second subsection is an extension of the ideas of \cite{DFO} and \cite{DGO}, where we construct generic Harish-Chandra modules for certain invariants of simple GWAs.

The third subsection discuss in the detail this construction in the case of the Weyl algebras.

Before we start, we revisit the paper \cite{FS4} and offer a new proof of the Bernstein inequality of invariants of tensorial GWAs of classical or quantum type; and in particular, invariants of the Weyl algebra.

We then show that the modules in the following three subsections are all holonomic.

In the case of the Weyl algebra $W_n(\k)$, the Bernstein inequality is the well known statement that, if $M$ is a finitely generated module for it, then $\GK \, M \geq n$ (see, e.g., \cite[Chapter 8]{KL}). For tensorial simple GWAs of classical or quantum type and arbitrary rings of differential operators on smooth affine varieties, see \cite{FS4} and references therein.

The main result about their Gelfand-Kirillov dimensions, which we recall here, is:

\begin{theorem}
\begin{enumerate}
    \item Let $A$ be an affine regular domain of finite Krull dimension $n$ over a field of characteristic $0$ (such as $\mathcal{O}(X)$, $X$ an smooth affine variety of dimension $n$). Then for every finitely generated module over $\mathcal{D}(A)$, the Gelfand-Kirillov dimension is greater then or equal $n$.
    \item
    The same result holds if $A$ is simple tensorial GWA of classical or quantum type of rank $n$.
\end{enumerate}   
\end{theorem}

In all cases, a non-zero finitely generated module $M$ is called \emph{holonomic}, following the terminology of the Weyl algebra (cf. \cite{KL}), if it has minimmal Gelfand-Kirillov dimension $n$. In order to the holonomic modules to form an abelian subcategory of modules, we determine that $0$ is a holonomic module.

\begin{lemma}
    Let $A$ be the rank $n$ Weyl algebra or some of the algebras of previous Theorem. Let $G$ be any finite group of automorphisms of $A$. If $M$ is any finitely generated $A*G$ module, then $\GK \, M \geq n$.
\end{lemma}
\begin{proof}
    $M$ is, in particular, a finitely generated $A$-module, so the result follows from the usual Bernstein inequality for $A$.
\end{proof}

\begin{proposition}\emph{(Bernstein inequality)}
    Let $M$ be any finitely generated $A^G$-module, where $A$ is an algebra of previous Lemma and $G$ a finite group of outer automorphisms. Then $\operatorname{GK} \, M \geq n$
\end{proposition}
\begin{proof}
    A combination of Theorem \ref{essential-invariant}, Theorem \ref{Morita}, and the previous Lemma.
\end{proof}

\subsection{Category O modules are Harish-Chandra and holonomic}

In this section, $\k=\C$.

Let $(G,h)$ be a complex reflection group representation, with $G=G(m,p,n)$, $t=1$. Let $S$ be the set of complex reflections of $G$ and $\mathfrak{c}:S \rightarrow \C$ a function invariant with respect to $G$-conjugation.

With this data, we can construct an algebra called rational Cherednik algbra, introduced in \cite{EG}. When $\mathfrak{c}=0$, the rational Cherednik algebra is just $W_n(\C)*G$. So we can use theory of rational Cherednik algebras to study $W_n(\C)*G$-modules and then, by Morita equivalence, $W_n(\C)^G$-modules. As we will notice, the fact that $\mathfrak{c}=0$ makes the theory rather simple. In this section we follow \cite{Bellamy0} and \cite{Vale}.

We will use freely both the notations $W_n(\C)*G$ and $\mathcal{D}(h)*G$.

\begin{theorem}\label{PBW-Cherednik}
    \emph{(PBW)}
The multiplication map $\C[h] \otimes \C G \otimes \C[h^*] \rightarrow \mathcal{D}(h)*G$ is an isomorphism.    
\end{theorem}
\begin{proof}
    \cite{EG}.
\end{proof}

\begin{remark}
    Remember that $\C[h]= \operatorname{Sym} \, h^*$ and $\C[h^*]= \operatorname{Sym}\, h$.
\end{remark}




\begin{definition}
    The category $\mathcal{O}$ for $W_n(\C)*G$ consists of the full subcategory of $W_n(\C)*G$-modules consisting of finitely generated modules $M$ where $h$ acts locally nilpotently: $y \in h$ and $m \in M$ implies that for $N>>0$, $y^N m=0$.
\end{definition}
 This is, more precisely, $\mathcal{O}_{W_n(\C)*G}(0)$, in the block decomposition of the category discussed in Theorem \ref{O-H-C} (cf. \cite{Berest}).

\begin{proposition}
    Category $\mathcal{O}$ is an abelian Jordan-Hölder subcategory.
\end{proposition}
\begin{proof}
    \cite[Proposition 1.24]{Vale}.
\end{proof}

\begin{definition}
    Let $\operatorname{Irrep}(G)$ denotes the isoclass of irreducible finite dimensional $G$-modules. Let $\tau \in \operatorname{Irrep}(G)$. We can make it a $\C[h^*]*G$-module by declaring that $h$ acts as 0. The standard module $M(\tau)$ is

    \[ M(\tau)=W_n(\C)*G \otimes_{\C[h*]*G} \tau \]
\end{definition}

\begin{lemma}\label{xyz} 
As a vector space, $M(\tau)=\C[h] \otimes \tau$.
\end{lemma}

\begin{proposition}
    $M(\tau)$ has a unique irreducible quotient $L(\tau)$. $\{ M(\tau), L(\tau) \}_{\tau \in \operatorname{Irrep}(G)}$ belongs to category $\mathcal{O}$, and $\{ L(\tau) \}_{\tau \in \operatorname{Irrep}(G)}$ accounts for all its irreducible objects.
\end{proposition}
\begin{proof}
\cite[Proposition 1.23]{Vale}
\end{proof}

\begin{theorem}
    Category $\mathcal{O}$ is semisimple and for each $\tau \in \operatorname{Irrep}(G)$ we have $M(\tau)=L(\tau)$.
\end{theorem}
\begin{proof}
    \cite[Theorem 2.1, Lemma 2.2]{Vale}, and now it is known that the hypothesis of this proof, that the dimension of the Hecke algebra is $|G|$, is always true, for all $G$ \cite{Etingof}.
\end{proof}
.






\begin{theorem}\label{O-is-Harish-Chandra}
   $M(\tau)$ is a Harish-Chandra module for $W_n(\C)*G$.
\end{theorem}
\begin{proof}
    As a vector space, $M(\tau)$ is isomorphic to $\C[h] \otimes \tau$.

    The Harish-Chandra subalgebra for $W_n(\C)*G$ is $\Gamma= \C[t_1,\ldots,t_n]$, $t_i=y_i x_i$, $y_1, \ldots, y_n$ a base for $h$ and $x's$ a dual basis in $h^*$. We can write $\Gamma=\C[u_1,
\ldots,u_n], \, u_i = x_i y_i,$ for $t_i-u_i=1$ and so the algebras coincide.
    Let's see how each $u_i$ acts on $M(\tau)$. So let $p(x) \otimes m \in M(\tau), \, p(x) \in \C[h], m \in \tau$.

    Write $p(x)= \sum_\alpha c_\alpha x^\alpha, \, c_\alpha \in \C$.

\[ [x_i y_i, x_j]=0, i \neq j; \, [x_i y_i,x_i]=x_i[y_i,x_i]+[x_i,x_i]y_i=x_i.\]


 \[ u_i.\big(p(x) \otimes m\big)= [u_i,p(x)] \otimes m + p(x) \otimes u_i.m=[u_i,p(x)] \otimes m.\]

  since $h$ acts by zero on $\tau$. In particular, \[u_i.(x_1^{\ell_1}x_2^{\ell_2}\cdots x_n^{\ell_n}\otimes\tau)=\ell_i \cdot x_1^{\ell_1}x_2^{\ell_2}\cdots x_n^{\ell_n}\otimes\tau.\]
 
\medskip

We have \[ M(\tau)=  \bigoplus_{ (\ell_1,\ldots,\ell_n) \in \mathbb{N}^n} \C x_1^{\ell_1}x_2^{\ell_2}\ldots x_n^{\ell_n} \otimes \tau \]

 Call $\bar{\ell}=(\ell_1,\ldots,\ell_n)$, and $\chi_{\bar{\ell}}: \Gamma \rightarrow \C$ the character of $\Gamma$ that sends each $u_i$ to ${\ell_i}$. Then $\Gamma$ acts with character $\chi_{\bar{\ell}}$ in $\C x_1^{\ell_1}x_2^{\ell_2}\ldots x_n^{\ell_n} \otimes \tau$.

\medskip

Hence \[ M(\tau)=  \bigoplus_{ (\ell_1,\ldots,\ell_n) \in \mathbb{N}^n} \C x_1^{\ell_1}x_2^{\ell_2}\ldots x_n^{\ell_n} \otimes \tau \]

is diagonalizable for the action of $\Gamma$. Hence the module is Harish-Chandra.

    \end{proof}
    
\begin{corollary}
   $M(\tau)^{G}$ is an irreducible Harish-Chandra module for $W_n(\C)*G$
\end{corollary}
\begin{proof}
    It follows from the previous Theorem and Theorems \ref{essential-invariant}, \ref{Morita2} and \ref{Morita-HC}.
\end{proof}


We can now offer a simple proof of the well known fact that every module in the category $\mathcal{O}$ for $W_n(\C)*G$ is holonomic. By Morita equivalence, a similar statement holds for $W_n(\C)^G$

\begin{proof}
    Since the category is semisimple and Jordan-Hölder, $M \in \mathcal{O}$ is a direct sum of a finite number of irreducible modules. By \cite[Proposition 5.1a)]{KL}, it is enough to show that each $L(\tau)=M(\tau)$ is holonomic. Let $W$ be a finite dimensional vector space that generates $M(\tau)$ as a $W_n(\C)*G$-module. Without loss of generality, we can assume $W$ of the form $U \otimes \tau$, where $U$ is a finite dimensional vector space that generates $\C[h]$ as a module over itself.
    Ler $V$ be a frame of $W_n(\C)*G$. There is no loss in generality assuming $V$ of the form $V= V' \otimes 1 \otimes \C G + 1 \otimes V'' \otimes \C G$, $V'$ a frame for $\C[h]$, $V''$ a frame for $\C[h^*]$, where we are using Theorem \ref{PBW-Cherednik}. $V^n W=V^{'n}.U \otimes \tau + U \otimes \tau$. It is clear then that the growth of the function  $d_{V, W}(n)$ is the same as $d_{V', U}(n)$ (see \cite[Proposition 8.1.7]{McConnell} and \cite[Proposition 2.1b0]{KL}). Since

    \[ \limsup_{n \to \infty} \, \log_n (d_{V',U}(n))=n,\]

\medskip

    $GK \, M(\tau)=n$, and hence every module in category $\mathcal{O}$ is holonomic.
    
\end{proof}

We have seen that the algebra $W_n(\C)*G$ has two natural Harish-Chandra subalgebras: $\C[h^*]$ (Theorem \ref{O-H-C}) and $\Gamma=\C[t_1, \ldots, t_n]$. So we can consider two module categories: $\mathcal{O}_{W_n(\C)*G}=\mathbb{H}(W_n(\C)*G, \C[h])$ and $\mathbb{H}(W_n(\C)*G,\C[t_1,\ldots,t_n])$, the which arises from the flag order realization of $W_n(\C)*G$ when $G=G(m,p,n)$. We have that the category $\mathcal{O}$ of \cite{GGOR} belongs to both.

\begin{problem}
    What is, precisely, the relation between $\mathcal{O}_{W_n(\C)*G}$ and the Harish-Chandra modules for $W_n(\C)*G$ with respect to $\Gamma$?
\end{problem}

\subsection{Generic modules for fixed rings of simple GWA}

We begin this section  with the assumption that $D$ is an affine commutative domain over an algebraically closed field $\k$ of characteristic 0, $D(a, \sigma)$ a rank n GWA.

\begin{lemma}
    $D(a,\sigma)*G$, where $G$ is a finite group of algebra automorphisms, has $D$ as Harish-Chandra subalgebra
\end{lemma}

 Suppose now that $D(a,\sigma)$ is the tensor product, $n$, times, of the same GWA of rank 1, and restrict to the case where $G\leq S_n$ acts on $D(a,\sigma)$ by permutations, like before:  $\pi X_i^{\pm}=X_{\pi(i)}^\pm; \pi (d_1 \otimes \ldots \otimes d_n) =(d_{\pi(1)} \otimes \ldots \otimes d_{\pi(n)}) $. The algebra $D(a,\sigma)*G$ has generators $D$, $X_i^+, X_i^-$, $i=1, \ldots, n$, and $g \in G$. The relations are the same as in the Definition \ref{def-GWA}, plus the two additional relations:

\[ \pi X_i^\pm \pi^{-1}= X_{\pi(i)}^\pm, i=1,\ldots,n, \pi \in G, \]
\[ \pi (d_1 \otimes \ldots \otimes d_n) \pi^{-1} =d_{\pi(1)} \otimes \ldots \otimes d_{\pi(n)}. \]

Given a maximal ideal of $D$, $\m$, denote by $f_\m:D/\m \rightarrow \k$ the canonical $\k$-algebra isomorphism. Let $\phi$ be any $\k$-algebra automorphism of $D$. 
Then it is easy to see that:

\[ (\dagger) \, f_{\phi^{-1} \m}(a+\phi^{-1}\m)= f_\m(\phi(a)+\m),\qquad\forall a\in D.\]

\medskip

Fix a maximal ideal $\m$ of $D$. Consider the symbols $T(\theta(\m))$, where $\theta$ runs throught all the elements of the group of automorphisms of $D$ generated the $\sigma_i$, $i=1, \ldots, n$ and $G$. We call those $T(.)$ \emph{generalized tableaux}. Let $T_\m$ be the vector space with basis the generalized tableaux.

For each basis element $T(\n) \in T_\m$, where $\n$ is an ideal of the form $\theta(\m)$, define the following linear actions:

 \[ z.T(\n)=f_\n(z+\n)T(\n); \]
 \[ X_i^+.T(\n) = T(\sigma_i \n); \]
 \[ X_i^- T(\n) = f_{\sigma_i^{-1}  \n}(a_i+\sigma_i^{-1}\n)T(\sigma_i^{-1} \n),\]
 \[ \pi .T(\n)= T(\pi \n)\]
 \[ z \in D, X_i^+, X_i^- \in D(a , \sigma), i=1,\ldots, n, \pi \in G \leq S_n.\] 

 \medskip

We need to check that the operators defined above satisfy the defining relations of $D(a,\sigma)*G$. This is easy to do, by repeated application of $(\dagger)$:
\begin{itemize}

\item $X_i^+ z T(\n)= f_n(z+\n)X_i T(\n)= f_\n(z) T(\sigma_i \n)=f_{\sigma_i \n}(\sigma_i(z) +\sigma_i(\n))T(\sigma_i(\n))= \sigma_i(z)X_i^+ T(\n),$
\item $X_i^- z T(\n) = f_\n(z+\n) f_{\sigma_i^{-1} n}(a_i+\sigma_i^{-1} \n) T(\sigma_i^{-1} \n) = $ 
\[ f_{\sigma_i^{-1} \n}(\sigma_i^{-1} z + \sigma_i^{-1} \n) f_{\sigma_i^{-1} \n}(a_i + \sigma_i^{-1} \n) T(\sigma_i^{-1} n)= \sigma_i^{-1} (z) X_i^- T(\n), \]
\item  $X_i^- X_i^+ T(\n)= f_{\n}(a_i+\n) T(\n)= a_i T(\n),$

\item $X_i^+ X_i^- T(\n)= f_{\sigma_i^{-1}\n}(a_i+\sigma_i^{-1}\n)T(\n)=f_\n(\sigma_i(a_i)+\n)T(\n)= \sigma_i(a_i) T(\n)$

\item $\pi X_i^+ \pi^{-1} T(\n)=T(\pi \sigma_i \pi^{-1} \n)=T(\pi \n)=X_{\pi(i)}^+ T(\n),$

\item $\pi X_i^-  \pi^{-1} T(\n)= f_{\sigma_i \pi^{-1} \n}(a_i + \sigma_i \pi^{-1} \n) \pi T(\sigma_i \pi^{-1} \n)=$
\[f_{ \pi \sigma_{\pi(i)} \pi^{-1} \n}(a_{\pi(i)} + \pi \sigma_i \pi^{-1} \n) T(\pi \sigma_i \pi^{-1} \n)= f_{\sigma_{\pi(i)}\n}(a_{\pi(i)} +\sigma_{\pi(i)} \n)T(\sigma_{\pi(i)} \n )= X_{\pi(i)}^- T(\n).  \]

\end{itemize}

\begin{theorem}\label{gentableaux}
For each $\m \in \operatorname{Specm} \, D$, there exists an irreducible $D(a, \sigma)*G$-module $T_\m$ which is a Harish-Chandra module with $\m$ in its support. If $D(a,\sigma)$ is a simple tensorial GWA of classical of quantum type, it is holonomic. In this case $T_\m^G$ is a simple Harish-Chandra module for $D(a,\sigma)^G$ that is also holonomic.
\end{theorem}
\begin{proof}
The fact that the module is $T_\m$ irreducible and has $\m$ in its support is clear by construction. The statement about $T_\m^G$ follows from  Theorems \ref{Morita} and \ref{Morita-HC}. $T(\m)$ as a single vector generates $T_m$. Consider the analogues of the Bernstein filtration $\mathfrak{B}=\{  F_n \}_{n \geq 0}$ for those generalized Weyl algebras (\cite[pp. 6]{FS4}). Let $\mathfrak{M}$ be $\langle \sigma_1, \ldots, \sigma_n \rangle \rtimes G \simeq \mathbb{Z} \rtimes G$, and let $d_n = \operatorname{dim} F_n T(\m)$. Since $\mathfrak{M}$ acts transitively on the generalized tableaux basis of $T_\m$, $\limsup_{n \to \infty} \log_n d_n = \operatorname{growth} \, \mathfrak{M} =n$, is also the value of $\GK \, T_m$. So we are done.
\end{proof}

\subsection{Generic modules for Weyl algebra invariants}

In this subsection we give a down-to-earth example of the construction in the previous section. We will consider modules for $W_n(\k)*{\mathcal{A}_n}$, for simplicity. All the modules in this section, despite very simple, are new.

Let $\mathfrak{L}= \operatorname{Specm} \, \Gamma$. By the Hilbert's Nullstellensatz, every element of $ \lambda \in \mathfrak{L}$ has the form $(t_1-\lambda_1, \ldots, t_n - \lambda_n), \lambda_i \in \k$. We, then, \emph{identify} $\lambda \in \mathfrak{L}$ with an $n$-tuple in $\k^n$, $(t_1-\lambda_1, \ldots, t_n - \lambda_n) \equiv (\lambda_1, \ldots, \lambda_n)$.

Let $\lambda \in \mathfrak{L}$. Let $\delta_i$ be the $n$-tuple with $1$ in entry $i$ and $0$ in all others. Calling $\varepsilon_1, \ldots, \varepsilon_n$ the standard basis of $\mathbb{Z}^n$, they act on $\lambda \in \mathfrak{L}$ as follows: $\varepsilon_i.\lambda=\lambda+\delta_i$.

We can also make $\mathcal{A}_n$ act on $\lambda$: if $\pi \in A_n$, $\pi.(\lambda_1,\ldots,\lambda_n)=(\lambda_{\pi(1)}, \ldots, \lambda_{\pi(n)})$. Hence the whole group $\mathfrak{M}:= \mathbb{Z}^n \rtimes \mathcal{A}_n$ acts on $\mathfrak{L}$.

Let $\beta \in \mathfrak{L}$ and $\mathfrak{B}$ be the orbit of $\beta$ in $\mathfrak{L}$ under the action of the group $\mathfrak{M}$. Let $T_\beta$ be the vector space with basis symbols of the form $T(\lambda)$, where $\lambda \in \mathfrak{B}$. The $W_n(\k)*A_n$-module structure on $T_\beta$, following the construction for Theorem \ref{gentableaux}, can be given quite explicitly:

\[ x_i . T(\lambda)= T(\lambda+ \delta_i), \]

\[ y_i . T(\lambda) = \lambda_i T(\lambda - \delta_i)\, \]

\[ \pi. T(\lambda)=T(\pi(\lambda)). \]

The next two results are restaments of Theorem \ref{gentableaux}, remembering that, by \cite[Theorem 5]{Levasseur}, $W_n(\k)^{\mathcal{A}_n} \simeq \mathcal{D}(\k^n/\mathcal{A}_n)$, as the alternating group contains no trivial reflections.

\begin{theorem}
    $\operatorname{GKdim} \, T_\beta=n$. That is, the module is holonomic.
\end{theorem}

\begin{corollary}
    $T_\beta^{A_n}$ is a holonomic module for $W_n(\k)^{A_n} \simeq \mathcal{D}(\k^n/A_n)$.
\end{corollary}

\section*{Acknowledgments}
The author would like to acknowledge the important role played by J. T. Hartwig, E. C. Jauch and V. Futorny in writing this paper, for many fruitful discussions, and also M. Lorenz, for his explanations about reflection groups in multiplicative invariant theory, P. Etingof, for some discussions related to rational Cherednik algebras, Ken Goodearl for presenting the author relevant papers on noncommutative lying over the related class of FCR rings, Santanu Tantubay for discussion of classical and recent results about Harish-Chandra modules for the Virasoro algebra, the Witt algebra, and other infinite dimensional simple Lie algebras of vector fields due to E. Cartan, and Dylan Fillmore for the discussion of many of the proofs of this manuscript. Most of this work was made at Shenzhen International Center for Mathematics in SUSTech, China, and the author would like to stress the outstanding research conditions, and all the faculty, post-docs and students that contribute to a very friendly atmosphere. In particular, the author would like to stress the kindness and helpfulness of our chinese staff: Qunwang Zhang, Zhiwang Fu, Weiyu Chen and Jialing Tang.



\end{document}